\documentclass[12pt, reqno]{amsart}
\usepackage{amssymb, mathrsfs}
\usepackage{latexsym}
\usepackage{amsmath}
\usepackage{amsthm}
\usepackage{amsfonts}
\usepackage{dsfont, upgreek}
\usepackage{mathtools}
\usepackage{epsfig}
\usepackage{amscd}
\usepackage{graphicx}

\usepackage[colorlinks=true,linkcolor=blue,citecolor=blue]{hyperref}

\usepackage{tikz-cd}

\usepackage{enumitem}
\usepackage{url}
\setlist[enumerate]{itemsep=0.5ex}

\usepackage{color}
\usepackage{tikz}
\usetikzlibrary{patterns}

\usepackage[left=1.2in,right=1.2in,top=1.4in,bottom=1.4in]{geometry}
\usetikzlibrary{decorations.markings}
\usetikzlibrary{arrows.meta}

\usepackage{extarrows}

\usepackage{adjustbox}
\usepackage{pgfplots}
\usepgfplotslibrary{colormaps}
\usepackage{slashed}

\usepackage[all,cmtip]{xy}
\usepackage{comment}

\theoremstyle{plain}
\newtheorem{theorem}{Theorem}[section]
\newtheorem{proposition}[theorem]{Proposition}
\newtheorem{lemma}[theorem]{Lemma}

\newtheorem{conjecture}[theorem]{Conjecture}

\theoremstyle{definition} 
\newtheorem{definition}[theorem]{Definition}
\newtheorem{example}[theorem]{Example}
\newtheorem*{claim*}{Claim}

\theoremstyle{remark} 
\newtheorem{remark}[theorem]{Remark}

\numberwithin{equation}{section}

\newcommand{\Sc}{\mathrm{Sc}}

 \newcommand{\cK}{\mathcal{K}}
 \newcommand{\dist}{\mathrm{dist}}
\newcommand{\R}{\mathbb{R}}
\newcommand{\diam}{\mathrm{diam}}
\newcommand{\Z}{\mathbb{Z}}
\newcommand{\rank}{\mathrm{rank}}
\newcommand{\Ric}{\mathrm{Ric}}
\newcommand{\lf}{\mathit{lf}}
\newcommand{\macro}{\mathrm{mac}}

\DeclareMathOperator{\width}{width}
\DeclareMathOperator{\fillrad}{FillRad}
\DeclareMathOperator{\inj}{Inj}
\DeclareMathOperator{\radsphere}{Rad_{\mathbb{S}^n}}

\DeclareMathOperator{\kr}{ker}
\DeclareMathOperator{\im}{Im}
\DeclareMathOperator{\lip}{Lip}
\DeclareMathOperator{\conj}{Conj}
\DeclareMathOperator{\conv}{Conv}

\newcommand{\interior}[1]{%
	{\kern0pt#1}^{\mathrm{\,o}}%
}

\makeatletter
\let\save@mathaccent\mathaccent
\newcommand*\if@single[3]{%
	\setbox0\hbox{${\mathaccent"0362{#1}}^H$}%
	\setbox2\hbox{${\mathaccent"0362{\kern0pt#1}}^H$}%
	\ifdim\ht0=\ht2 #3\else #2\fi
}
\newcommand*\rel@kern[1]{\kern#1\dimexpr\macc@kerna}
\newcommand*\wideaccent[2]{\@ifnextchar^{{\wide@accent{#1}{#2}{0}}}{\wide@accent{#1}{#2}{1}}}
\newcommand*\wide@accent[3]{\if@single{#2}{\wide@accent@{#1}{#2}{#3}{1}}{\wide@accent@{#1}{#2}{#3}{2}}}
\newcommand*\wide@accent@[4]{%
	\begingroup
	\def\mathaccent##1##2{%
		\let\mathaccent\save@mathaccent
		\if#42 \let\macc@nucleus\first@char \fi
		\setbox\z@\hbox{$\macc@style{\macc@nucleus}_{}$}%
		\setbox\tw@\hbox{$\macc@style{\macc@nucleus}{}_{}$}%
		\dimen@\wd\tw@
		\advance\dimen@-\wd\z@
		\divide\dimen@ 3
		\@tempdima\wd\tw@
		\advance\@tempdima-\scriptspace
		\divide\@tempdima 10
		\advance\dimen@-\@tempdima
		\ifdim\dimen@>\z@ \dimen@0pt\fi
		\rel@kern{0.6}\kern-\dimen@
		\if#41
		#1{\rel@kern{-0.6}\kern\dimen@\macc@nucleus\rel@kern{0.4}\kern\dimen@}%
		\advance\dimen@0.4\dimexpr\macc@kerna
		\let\final@kern#3%
		\ifdim\dimen@<\z@ \let\final@kern1\fi
		\if\final@kern1 \kern-\dimen@\fi
		\else
		#1{\rel@kern{-0.6}\kern\dimen@#2}%
		\fi
	}%
	\macc@depth\@ne
	\let\math@bgroup\@empty \let\math@egroup\macc@set@skewchar
	\mathsurround\z@ \frozen@everymath{\mathgroup\macc@group\relax}%
	\macc@set@skewchar\relax
	\let\mathaccentV\macc@nested@a
	\if#41
	\macc@nested@a\relax111{#2}%
	\else
	\def\gobble@till@marker##1\endmarker{}%
	\futurelet\first@char\gobble@till@marker#2\endmarker
	\ifcat\noexpand\first@char A\else
	\def\first@char{}%
	\fi
	\macc@nested@a\relax111{\first@char}%
	\fi
	\endgroup
}
\makeatother

\newcommand\overbar{\wideaccent\overline}

\makeatletter
\newcommand*{\transpose}{%
	{\mathpalette\@transpose{}}%
}
\newcommand*{\@transpose}[2]{%
	\raisebox{\depth}{$\m@th#1\intercal$}%
}
\makeatother

\begin{document}
	
	\title{Filling Radius, Quantitative $K$-theory and Positive Scalar Curvature }
		\author{Jinmin Wang}
	\address[Jinmin Wang]{Department of
		Mathematics, Texas A\&M University}
	\email{jinmin@tamu.edu}
	%
	\author{Zhizhang Xie}
	\address[Zhizhang Xie]{ Department of Mathematics, Texas A\&M University }
	\email{xie@math.tamu.edu}
	\thanks{The second author is partially supported by NSF 1952693 and 2247322.}
	\author{Guoliang Yu}
	\address[Guoliang Yu]{ Department of
		Mathematics, Texas A\&M University}
	\email{guoliangyu@math.tamu.edu}
	\thanks{The third author is partially supported by NSF 2000082, 2247313 and the Simons Fellows Program.}
		\author{Bo Zhu}
	\address[Bo Zhu]{Department of Mathematics, Texas A\&M University }
	\email{bozhu@tamu.edu}
	%
	
	
 \begin{abstract}
 	We prove a quantitative upper bound on the filling radius of complete, spin manifolds with uniformly positive scalar curvature using the quantitative operator $K$-theory and index theory.
 \end{abstract}

	\maketitle

\section{Introduction}

In this paper, we study the following filling radius conjecture of Gromov. 
\begin{conjecture}[\cite{gromovmetricstructure}]\label{conj: gromov_filling_radius_conjecture} 
	If $(M^n,g)$ is a complete Riemannian manifold with uniformly positive scalar curvature $\Sc_g \geq \sigma^2 >0$, then there exists a universal constant $c_n$ depending only on $n$ such that
	\begin{equation}
		\fillrad(M) \leq \frac{c_n}{\sigma}.
	\end{equation}

\end{conjecture}
Here $\fillrad(M)$ stands for the filling radius of $M$, namely the smallest radius of the neighborhood of $M$ inside any isometric embedding, in which $M$ is homologous to zero. See Definition \ref{def: filling_radius} for the precise definition of filling radius. 
A manifold  $M$ is said to be aspherical if $\pi_k(M)=0$ for any $k \geq 2$.
Since the universal cover of any closed aspherical manifold has infinite filling radius (see Example \ref{example: ashperical_filling} for more details), Conjecture \ref{conj: gromov_filling_radius_conjecture} implies the Gromov--Lawson conjecture \cite{gl_psc_dirac,Rosenberg_psc3} that any closed aspherical manifold admits no complete Riemannian metric with positive scalar curvature. 
Note that Conjecture \ref{conj: gromov_filling_radius_conjecture} is weaker than Gromov's Uryson width conjecture \cite[Section 3]{Gromovadozen} (also see Conjecture \ref{conj: gromov_uryson_width_conjecture} in the Appendix of this paper).

Conjecture \ref{conj: gromov_filling_radius_conjecture} has been proved for $n=3$ \cite{gl_psc_dirac, Gromov4lectures2019,liokumovich2023waist}. It remains open for $n \geq 4$ in general. The main purpose of this paper  is to employ the quantitative operator $K$-theory and index theory to prove a large scale version of  Conjecture \ref{conj: gromov_filling_radius_conjecture} for manifolds with finite asymptotic dimension.

\begin{theorem}\label{thm:main}
	Let $(M^n,g)$ be a complete, noncompact, spin Riemannian manifold with finite asymptotic dimension $m$ with control function $\mathcal{D}$. If $(M,g)$ has uniformly positive scalar curvature $Sc_g\geq \sigma^2 >0$ and bounded geometry, then there exists a constant $c=c(\sigma, m, \mathcal{D})$ such that
	\begin{equation}
		\fillrad_{\mathbb{Q}}(M) \leq c.
	\end{equation}
\end{theorem}
From the proof of Theorem \ref{thm:main}, the constant $c$ is explicitly calculated as
$$c\leq \mathcal{D}\left(\frac{A e^{Bm}}{{\sigma}}\right),$$
where $A$ and $B$ are some universal positive constants. Recall that  a complete, non-compact Riemannian manifold $(M^n,g)$ is said to have asymptotic dimension $m$ with diameter control $\mathcal{D}$ if $m$ is the minimal integer such that for any $r>0$, $M$ admits a uniformly bounded cover $\mathcal{C} = \{U_i\}_{i \in I}$ such that (1)		any  ball $\mathcal{B}(p,r) = \{x \in M, \dist(x,p) < r\}$ intersects with at most $m+1$ members $\{U_{i_1}, \cdots, U_{i_m}\}$ in $\mathcal{C}$;  (2) $\displaystyle \sup_{i \in I}\diam(U_i) \leq \mathcal{D}(r)$. Here  $\mathcal{D}: \mathbb{R}^+ \rightarrow \mathbb{R}^+$ is a non-deceasing function with $\mathcal{D}(r) \geq r$.

The notion of asymptotic dimension was introduced by Gromov \cite{gromov_asy_invariant}. The third author developed a quantitative operator K-theory and proved that any finitely generated group with finite asymptotic dimension satisfies the  strong Novikov conjecture  \cite{Yu_Novikov_fsd}. 
The strong Novikov conjecture provides an algorithm for determining  whether the  higher index of an elliptic differential operator vanishes or not. The connection between 
the vanishing of higher index and positive scalar curvature was established by Rosenberg
\cite{Rosenberg_psc1, Rosenberg_psc2, Rosenberg_psc3}. We refer the reader to  \cite{zhang_psc_foliations, twy_Roe, gmw_psc_duality, ww_localized_index, Wang:2021tq,Wang:2021uj,Wang:2021um, Wang:2022vf} for some of the further developments on the (higher) index theory, in particular for spaces with singularities,  and its applications to scalar curvature problems.

 In this paper, the estimate of the filling radius in  Theorem \ref{thm:main} is obtained by studying the pairing 
of  the Dirac operator and  $K$-theory classes
with a certain quantitative Lipschitz control. The $K$-theory with Lipschitz control was initially developed to study the scalar curvature decay in  \cite{Wang:2021uj}. Here is an outline of the proof of Theorem \ref{thm:main}. If $M$ has finite asymptotic dimension $m$, then for any $r>0$, there exists an open cover with $r$-multiplicity bounded by $m+1$. The natural map from $M$ to the corresponding nerve complex $\mathcal N_r$ of this open cover has a small Lipschitz constant when $r$ is large.  Now if the filling radius of $M$ is too large, then the $K$-homology class of the Dirac operator on $M$ does not map to zero in the $K$-homology of the nerve complex $\mathcal N_r$ for some large $r$. This implies that the pairing between the Dirac operator $D$ on $M$ and the pullback of a certain $K$-theory class of $\mathcal N_r$ is not zero. However, by the Lipschitz controlled $K$-theory, the pullback of this $K$-theory class has a small Lipschitz constant when $r$ is large. By a standard curvature estimate,  the uniformly positive scalar curvature on $M$ implies that the pairing of $D$ with any vector bundle with small Lipschitz constant is zero. This leads to a contradiction, hence the filling radius cannot be too large under the uniformly positive scalar curvature condition. Our estimates in fact produce an explicit upper bound on the filling radius.


The paper is organized as follows.  
In Section \ref{sec:fillingradius}, we review the notion of  filling radius and  some of its basic properties. In Section \ref{section: asym}, we review the notion of asymptotic dimension and Lipschitz controlled $K$-theory. In Section \ref{sec:main}, we prove the main theorem (Theorem \ref{thm:main}). In the appendix, to put the main results of the paper in a broader context,  we review a list of  metric invariants of Riemannian manifolds and their connections with positive scalar curvature.

\section{A basic introduction to Filling Radius}\label{sec:fillingradius}

In this section, we review the concept of filling radius introduced by Gromov. We recall some of the  basic properties of the filling radius and its connections to coarse geometry. Unless otherwise specified,  all Riemannian manifolds in this article are assumed to be smooth, connected and orientable.

\subsection{Kuratowski embedding}
In order to define the filling radius,  we first review the Kuratowski embedding, which isometrically embeds any given metric space into some Banach space.

\begin{definition} \label{def:Kuratowski_embedding}
	Suppose that $(X,d)$ is a metric space and $x_0$ is a point in $X$. Let $L^\infty(X)$ be the Banach space of all bounded functions on $X$ with the supremum norm $\| \ \|_{\infty}$. We define
	$\Psi \colon  (X,d) \rightarrow (L^\infty(X), \|\ \|_{\infty})$ by
	$${ \Psi (x)(y)=d(x,y)-d(x_{0},y),\quad \textup{ for all } x,y\in X}.$$
	The map $\Psi$ is called a Kuratowaki embedding of $X$.
\end{definition}

Note that this embedding depends on the choice of the based point $x_0$ and is therefore not entirely canonical. It is not difficult to see that  $\Psi$ is an isometric embedding.  Indeed,  on the one hand, for any $x_1, x_2 \in X$, by the triangle inequality, we have $$\|(\Psi(x_1)- \Psi(x_2)\|_{\infty} \leq d(x_1, x_2). $$
On the other hand, 
$$
\begin{array}{lcl}
	d(x_1, x_2)& = & d(x_1, x_2)-d(x_2,x_2)\\
	& =& \Psi(x_1)(x_2) -  \Psi(x_2)(x_2)\\
	&\leq& \sup_{y \in X} |\Psi(x_1)(y) - \Psi(x_2)(y)| \\
	&\leq& \|\Psi(x_1) - \Psi(x_2)\|_{\infty}.
\end{array}
$$
Consequently,  we have 
\begin{equation}
	d(x_1,x_2) = \|\Psi(x_1) - \Psi(x_2)\|_{\infty}
\end{equation}
for all $x_1,x_2 \in X$. From now on, we shall identify $X$ with $\Psi(X)$,  and write $X$ for both $X$ and $\Psi(X)$, since no confusion is likely to arise. In particular, we will view $X$ as a subspace  of $L^\infty(X)$ and denote the inclusion by $i\colon X\hookrightarrow L^\infty(X)$. 

If $(X,d)$ is a compact metric space, then there is a more canonical way of defining the embedding $\Psi$. Indeed, in this case, we may define 
$$\Psi \colon (X,d) \rightarrow L^\infty(X)$$
by
$$\Psi(x)(y) = d(y,x), \quad \forall x, y \in X.$$



\subsection{Filling radius} In this subsection, we  review the locally finite singular homology theory,   the notion of filling radius and some of its basic properties. 

Let $(X,d)$ be a complete metric space and $F$ a ring. We denote by $C_k^{\lf}(X; F)$ the abelian group of  formal sums (possibly infinite) of singular $k$-simplices: 
\begin{equation}
	c= \sum_{\sigma} r_{\sigma} \sigma,
\end{equation}
where $\sigma$ runs over all continuous maps from the standard $k$-simplex to $X$ and $r_\sigma\in F$, such that for any given bounded set $K \subset X$, there are only finitely many $\sigma$ in $c$ that have nonempty intersection with $K$. In this paper, we will mainly focus on the case where  $F$ equals  the integers $\mathbb{Z}$ or the rational numbers $\mathbb{Q}$. The usual definition of the boundary map $\partial $ for  singular chains makes these abelian groups into a chain complex:
\begin{equation}
	\cdots \rightarrow C^{\lf}_{k}(X; F) \rightarrow \cdots \rightarrow  C^{\lf}_{2}(X; F)
	\rightarrow C^{\lf}_{1}(X; F)\rightarrow C^{\lf}_{0}(X; F)\rightarrow 0.
\end{equation}

The locally finite $k$-th homology group $H_k^{\lf}(X; F)$ of $X$ is defined to be the $k$-th homology group of this  complex chain, that is, 
\begin{equation} \displaystyle	H_k^{\lf}(X; F) = \frac{\kr(\partial\colon  C^{\lf}_{k}(X; F) \rightarrow C^{\lf}_{k-1}(X; F))}{\im(\partial\colon  C^{\lf}_{k+1}(X; F) \rightarrow C^{\lf}_{k}(X; F))}.
\end{equation}
If $(X,d)$ is compact, then the locally finite homology groups coincide with the usual singular homology groups. Note that for each oriented $n$-dimensional manifold $M$, the fundamental class $[M]$ of $M$ generates $ H_n^{\lf}(M; F) \cong F$.

\begin{definition} [Gromov \cite{gromovfilling}]  \label{def: filling_radius}
	Suppose that $(M^n,g)$ is a Riemannian manifold and $\Psi\colon M \rightarrow L^\infty(M)$ is a Kuratowski embedding as given in Definition \ref{def:Kuratowski_embedding}. The filling radius of $M$ (with respect to $F$-coefficients)  is defined to be 
	\begin{equation}
		\fillrad_F(M) \coloneqq  \inf_ R\{ R : \Psi_*([M]) = 0\in H_n^{\lf}(N_R(\Psi(M)); F)\}, 
	\end{equation}
	where $N_R(\Psi(M))= \left\{x \in L^\infty(M) : 
	\mathrm{dist}(x, \Psi(M)) \leq R \right\}$ is the $R$-neighborhood of $\Psi(M)$ in $L^\infty(M)$. In the case $F= \mathbb Z$, we shall simply write $\fillrad(M)$ in place of $\fillrad_{\mathbb{Z}}(M)$.
\end{definition}

It is clear that  $\fillrad_{F}(M)$ is independent of the choice of the base point $x_0 \in M$ in Definition \ref{def: filling_radius}. One can show that if $(M^m,g)$ and $(N^n,h)$ are two Riemannian manifolds that  are coarsely equivalent, then $\fillrad_F(M) = \infty$ if and only if $\fillrad_F(N) =\infty$. 

\begin{lemma}\label{lemma:filling_radius_decreasing}
	Let $(M_1^n,g_1)$ and $(M_2^n,g_2)$ be Riemannian manifolds. If $f\colon (M_1,g_1) \rightarrow (M_2, g_2)$ is a distance-decreasing map with $\deg(f) \neq 0$, then
	\begin{equation}
		 \fillrad_{F}(M_1) \geq \fillrad_{F}(M_2).
	\end{equation}
\begin{proof}
We assume that both $(M_1, g_1)$ and $(M_2, g_2)$ are compact Riemannian manifolds, and the non-compact case can be proved similarly. We write $R_1:=\fillrad_{F}(M_1)$ and $R_2:=\fillrad_{F}(M_2)$. Let $\Psi_i$ be the Kuratowaki embedding of $M_i$ for $i=1,2$ in Definition \ref{def:Kuratowski_embedding}. Since $f: (M_1, g_1) \rightarrow (M_2,g_2)$ is a distance-decreasing map, we naturally obtain a distance-decreasing map $h = \Psi_2\circ f : M_1 \rightarrow L^\infty(M_2)$.


We extend $h$ to $H: L^\infty(M_1) \rightarrow L^\infty(M_2)$ by defining
$$\big(H(\varphi_1)\big)(m_2) = \inf_{m_1 \in M_1} \left((h(m_1))(m_2) + \|\varphi_1- \Psi_1(m_1)\|_{L^\infty(M_1)}\right)$$
for any $\varphi_1\in L^\infty(M_1)$ and  $m_2 \in M_2$.
By the construction, we have 
$$\|H(\varphi_1) - H(\psi_1)\|_{L^\infty(M_2)} \leq  \|\varphi_1 - \psi_2\|_{L^\infty(M_1)}$$
 for any $\varphi_1, \psi_1 \in L^\infty(M_1)$, and $H(m_1) = h(m_1)$ for any $m_1 \in M$. It follows that for any $R>0$, $H(N_R(M_1))\subset N_R(M_2)$.

By the definition of $R_1\coloneqq\fillrad_{F}(M_1)$, for any $\varepsilon > 0$, there exists an $(n+1)$-chain $c$ of $ N_{R_1+\varepsilon}(M_1) \subset L^\infty(M_1)$  such that $\partial c = (\Psi_{1})_*([M_1])$. It follows that 
$$\partial (H_*(c)) = H_*(\partial c) =  (\Psi_{2})_*f_*([M_1]) = \deg f\cdot (\Psi_{2})_*([M_2])$$
in the homology group of $N_{R_1+\varepsilon}(M_2)$. Consequently, $R_2 \leq R_1 + \varepsilon$. We conclude that $R_1 \geq R_2$ by letting $\varepsilon\to 0$.
\end{proof}
\end{lemma}

The following lemma will be useful in the proof of our main theorem. More precisely, it will allow us to reduce the proof for the odd dimensional case of our main theorem to the even dimensional case. 
	\begin{lemma} \label{remark: fillrad_of_direct_roduct}
		 If $(M^n,g)$ is a complete Riemannian manifold, then
	$$\fillrad_{\mathbb{Q}}(M\times\R) =  \fillrad_{\mathbb{Q}}(M).$$
	\begin{proof}
		For any $R>0$,  we set
		$$X_1=\bigcup_{k\in\Z} N_R(M\times[4R\cdot 2k,4R(2k+2)]),$$
		$$X_2=\bigcup_{k\in\Z}N_R(M\times[4R(2k-1),4R(2k+1)]).$$
		Note that
	$$X_1\cap X_2=\bigcup_{k\in\Z} \left(N_R(M\times[4R(k-2),4Rk])\cap N_R(M\times[4Rk,4R(k+2)])\right).$$
	Let $i_1\colon X_1\hookrightarrow N_R(M\times\R)$, $i_2\colon X_2\hookrightarrow  N_R(M\times\R)$ and $(j_1,j_2)\colon X_1\cap X_2\hookrightarrow X_1\times X_2$ be the inclusion map. Note that each component of $X_1, X_2, X_1 \bigcap X_2$ is homotopic to $N_R(M)$ by a bounded homotopy. Thus the locally finite homology groups  of $X_1$, $X_2$, and $X_1\cap X_2$ are all isomorphic to $\prod_\Z H^{\lf}_*(N_R(M);\mathbb Q)$. For notational simplicity, we omit the coefficient $\mathbb Q$ in the rest of the proof. 
		
		
		Consider  the following exact sequence:
		$$\cdots \to H^{\lf}_n(X_1\cap X_2)\to H^{\lf}_n(X_1)\oplus H^{\lf}_n(X_2)\to H^{\lf}_n(N_R(M\times\R)).$$
		More explicitly, we have
		$$\prod_\Z H^{\lf}_n(N_R(M))\xlongrightarrow{(j_{1*},j_{2*})} \prod_\Z H^{\lf}_n(N_R(M))\oplus \prod_\Z H^{\lf}_n(N_R(M)) \xlongrightarrow{i_{1*}-i_{2*}} H^{\lf}_n(N_R(M\times\R)).$$
		Denote by  $[M\times \R]$ the fundamental class on $M\times\R$, which is explicitly given by
		$$[M\times \R]=\sum_{k\in\Z} [M]\times[4Rk,4R(k+2)].$$
		By the above exact sequence, the class $\Psi_{M\times\R}([M\times\R])\in H^{\lf}_n(N_R(M\times\R))$ lies in the image of the map $i_{1*} - i_{2*}$. More explicitly, we have
		$$\Psi_{M\times\R}([M\times\R])=i_{1*}\big(\prod_\Z \Psi_{M}([M])\big)-i_{2*}\big(\prod_\Z(-\Psi_{M}([M]))\big).$$
		Therefore, if $[M]$ vanishes in $H^{\lf}_*(N_R(M))$, then $\Psi_{M\times\R}([M\times\R])$ vanishes in $H^{\lf}_n(N_R(M\times\R))$ as well. Consequently,
		we have $$\fillrad_{\mathbb Q}(M\times\R)\leq \fillrad_{\mathbb Q}(M).$$
		We remark that so far this part of the proof works for the filling radius with arbitrary coefficients. 
		
		Conversely, if $\Psi_{M\times\R}([M\times\R]) = 0$, then 
		$$i_{1*}\big(\prod_\Z \Psi_{M}([M])\big)-i_{2*}\big(\prod_\Z(-\Psi_{M}([M]))\big) = \Psi_{M\times\R}([M\times\R]) = 0.$$
		It follows from the above exact sequence that there exist  $c_k\in H^{\lf}_*(N_R(M))$ for $k\in\Z$ such that
		$$j_{1*}(\prod_\Z c_k)=\prod_\Z\Psi_{M}([M]) \textup{ and }  j_{2*}(\prod_\Z c_k)=\prod_\Z(-\Psi_{M}([M])).$$
		Note that each component of  $X_1\cap X_2$ only intersects  with two components of $X_1$ and  $X_2$, respectively. Therefore, for each $k\in\Z$, we have
		$$\Psi_{M}([M])=c_k=-\Psi_{M}([M]),$$ which implies that $2\Psi_{M}([M])=0$. It follows that $\Psi_{M}([M])=0$, since we are working with the coefficient $\mathbb Q$. To summarize, we have shown that if $\Psi_{M\times\R}([M\times\R])$ vanishes in $H^{\lf}_n(N_R(M\times\R))$, then $[M]$ vanishes in $H^{\lf}_*(N_R(M))$.	Consequently, we have 	$$\fillrad_{\mathbb Q}(M\times\R)\geq \fillrad_{\mathbb Q}(M).  $$
		This together with the inequality 
		$$\fillrad_{\mathbb Q}(M\times\R)\leq \fillrad_{\mathbb Q}(M)$$
		from the first part  completes the proof. 
	\end{proof}
\end{lemma}

\begin{remark}
	For the filling radius with $\mathbb Z_2$ coefficients, Gromov proved that  
	\begin{equation} 
		\displaystyle
		\fillrad_{\mathbb{Z}_2}(M_1 \times M_2) = \min\{\fillrad_{\mathbb{Z}_2}(M_1), \fillrad_{\mathbb{Z}_2}(M_2)\}.
	\end{equation}
	for any two Riemannian manifolds $(M_1, g_1)$ and $(M_2, g_2)$ \cite{gromovfilling}.
\end{remark}

The following example shows that  the filling radius of the universal cover of an aspherical manifolds is $\infty$.

\begin{example}\label{example: ashperical_filling}
	Let $(M^n,g)$ be a closed aspherical manifold, then the universal cover $\widetilde{M}$ of $M$ has $\fillrad_{\mathbb{Q}}(\widetilde{M}) = \infty$.
	\begin{proof}
		We begin by proving $\widetilde M$ is uniformly contractible, that is,  for any  $r >0$, there exists $R=R(r) \geq r$ such that $B(y,r)$ is contractible in $B(y, R)$ for any $y \in M$. Fix $x_0\in \widetilde M$. Let $d = \diam(M)$. Then,   for each $B(y, r)$,  there exists $g\in \pi_1(M)$ such that $g B(y, r) \subset B(x_0, d+2r)$. Note that $M$ admits a finite $CW$ complex structure, which naturally lifts to a $CW$ complex structure on $\widetilde M$. The ball $B(x_0, d+2r)$ is contained in some \emph{finite} $CW$ subcomplex $C$ of $\widetilde M$. Since $\widetilde M$ is contractible, it follows by induction on the dimensions of cells that each cell in $C$ contracts to a single point in $B(x_0, d+2r+s)$ for some $s>0$. Consequently, we have shown that $B(y,r)$ is contractible in $B(y,R)$ with $R(r)=d+2r+s$.
		
		Now let us  show that $\fillrad_\mathbb{Q}(\widetilde{M}) = \infty$. Assume to the contrary that $\fillrad_\mathbb{Q}(\widetilde{M}) <\infty$, then there exist $r >0$ and  an $(n+1)$ dimension simplicial complex $P$ such that $\partial P = \widetilde M$ and 
		\[ \widetilde{M} \subset P \subset U_{r}(\widetilde{M})\subset L^\infty(\widetilde{M}). \]
		However, since $\widetilde M$ is uniformly contractible, we see that the identity map $i\colon \widetilde M \to \widetilde M$ extends to a continous map $\varphi\colon P \to \widetilde M$ (see for example the proof of Theorem \ref{thm: inj_filling} for more details). It follows that $[\widetilde M] = [\partial P] = 0$ in $H^{\lf}_n(\widetilde M)$. This contradicts the fact $[\widetilde M]$ is the fundamental class and does not vanish in $H^{\lf}_n(\widetilde M)$. Therefore, we have shown that $\fillrad_\mathbb{Q}(\widetilde{M}) = \infty$. This finishes the proof.
	\end{proof}
\end{example}

\subsection{Filling radius and coarse geometry}

In this subsection, we review some basic results in coarse geometry and their connections to filling radius. A discrete metric space $\Gamma$ is said to be locally finite if every metric ball of finite radius in $\Gamma$ contains only finitely many elements of $\Gamma$. 

\begin{definition}[{cf. \cite{Roecoarse}}] Suppose $\Gamma$ is a locally finite discrete metric space. 
	Given any $d >0$, the Rips complex $P_d(\Gamma)$ is a simplical complex given as follows:  the set of vertices is $\Gamma$ and a finite subset $\{x_1, \cdots, x_n\} \subset \Gamma$ spans a simplex if and only $d(x_i, x_j) \leq d$ for all $1 \leq  i, j \leq n$. We define  the $k$-th coarse homology group $HX_k(\Gamma; F)$ of $\Gamma$ with coefficients in $F$ to be 
	\begin{equation}
		HX_k(\Gamma; F) = \lim_{d \rightarrow \infty} H_k^{\lf}(P_d(\Gamma); F).
	\end{equation}
\end{definition}

Now suppose  $(M^n, g)$ is a complete  Riemannian manifold with bounded geometry, where the latter condition means that  $(M^n, g)$ has uniformly bounded curvature and positive injectivity radius. A subspace  $\Gamma \subset M$  is said to be a net in $M$ if there exists $\delta > 0$ and $c >0$ such that
\begin{itemize}
	\item $d_g(x,y) \geq c$ for any distinct $x ,y \in \Gamma$;
	\item $d_g(x, \Gamma) \leq \delta $ for any $x \in M$.
\end{itemize}
Since $M$ has bounded geometry, there exists a triangulation of $M$ such that for each $r>0$,  the number of simplices in the metric ball $B(x, r)$ of radius $r$ is uniformly bounded for all $x\in M$. In particular, the set of vertices of this triangulation gives a net of $M$. 

We define the coarse homology groups of  Riemannian manifolds as follows.
\begin{definition}
	Suppose that $(M^n,g)$ is a complete Riemannian manifold with bounded geometry and $\Gamma$ is a net in $M$.  The $k$-th coarse homology group $HX_k(M;  F)$ of $M$ with coefficients in $F$ is defined to be $HX_k(\Gamma; F)$ for each $k \in \mathbb{N}$.
\end{definition}

 With $(M^n, g)$ and $\Gamma$ as above, then there  exists $d_0 > 0$ such that $M \subset P_{d_0}(\Gamma)$. Let $j\colon  M \rightarrow P_{d_0}(\Gamma)$ be the inclusion map. Then $j$ naturally induces a homomorphism 
$$j_*\colon  H_k^{\lf}(M; F) \rightarrow HX_k(M; F).$$

\begin{definition}
	Suppose that $(M^n,g)$ is a complete Riemannian manifold with bounded geometry. Let $[M] \in H_{n}^{\lf}(M; \mathbb{Q})$ be the fundamental class of $M$. We say $M$ is  rationally macroscopically large if $j_*([M]) \neq 0$ in $HX_n(M; \mathbb{Q})$. Otherwise, we say $M$ is  rationally macroscopically small. Similarly, by working with $\mathbb Z$ coefficients, we may define integral macroscopical largeness  and integral macroscopical smallness. 
\end{definition}

	The coarse $K$-homology $KX_\ast(M)$ of $M$ is defined to be 
	 \[ KX_\ast(M) = \lim_{d \rightarrow \infty} K_\ast(P_d(\Gamma)).\]  The Connes-Chern character map induces the following rational isomorphism\footnote{This isomorphism can for example be proved by a Mayer--Vietoris sequence argument.} between (locally finite) $K$-homology and locally finite homology theory:
\begin{equation*}
    K_0(M) \otimes\mathbb Q \cong \oplus_{n=0}^\infty  H^{\lf}_{2n}(M; \mathbb{Q})\textup{ and } \ K_1(M) \otimes\mathbb Q \cong  \oplus_{n=0}^\infty  H^{\lf}_{2n+1}(M; \mathbb{Q})
\end{equation*}
and similarly 
\[   KX_0(M) \otimes\mathbb Q \cong \oplus_{n=0}^\infty  HX_{2n}(M; \mathbb{Q})\textup{ and } \ KX_1(M) \otimes\mathbb Q \cong  \oplus_{n=0}^\infty  HX_{2n+1}(M; \mathbb{Q}). \]
In particular, if $M$ is spin and the $K$-homology class of its Dirac operator vanishes in $KX_\ast(M)$, then $M$ is rationally macroscopically small. This gives a useful $K$-homological  way to decide when $M$ is \emph{rationally} macroscopically  large or small via the techniques developed in the context of  the coarse Novikov conjecture.  We would like to remark that at the moment it is still unclear whether there is a similar $K$-homological approach to determining when $M$ is \emph{integrally} macroscopically  large or small. In fact, there  exist manifolds $(M^n,g)$ with $\fillrad(M) = \infty, \width_{n-1}(M) = \infty$ but $\fillrad_{\mathbb{Q}}(M) < \infty$ \cite{MR2893542}. Here $\width_{n-1}(M)$ is the $(n-1)$-th Uryson width of $M$. 
	See Definition \ref{def: width} for  the precise definition of Uryson width $\width_k(M)$ and its connection with  the filling radius. 

We conclude the section with the following theorem of   Gong and the third author  \cite[Theorem 4.1]{Gong_Yu_volume_growth_psc}. 
\begin{theorem}[{\cite[Theorem 4.1]{Gong_Yu_volume_growth_psc}}] \label{gong_yu_psc}
	Suppose that $(M^n,g)$ is a complete Riemannian manifold with uniformly positive scalar curvature, bounded geometry and sub-exponential volume growth, we have
	\[\fillrad_{\mathbb{Q}}(M,g) < \infty.\]
\end{theorem}
This result provides  a qualitative characterization of the filling radii for a class of  Riemannian manifolds with uniformly positive scalar curvature. The main purpose of the present paper is to prove a \emph{quantitative} characterization of the filling radii for a large class of Riemannian manifolds with uniformly positive scalar curvature in the sense of Conjecture \ref{conj: gromov_filling_radius_conjecture}.

\section{Lipschitz properties of spaces with finite asymptotic dimension} \label{section: asym}

In this section, we recall the notion of asymptotic dimension. We also review the Lipschitz controlled $K$-theory from \cite{Wang:2021uj}, which is a key ingredient for our proof to the main theorem of the paper. 

\begin{definition}\label{def: cover_nerve}
	Suppose that $(M,g)$ is a complete Riemannian manifold and $\mathcal{C} = \{U_i\}_{i \in I}$ is an open cover of $M$. The nerve $\mathcal{N}$ of $M$ is  a simplicial complex with one vertex $v_i$ for each open set $U_i$ with the property: given any $k$ tuple of vertices $v_{i_0}, \cdots \cdot, v_{i_k}$, there exists a simplex in the nerve with these given vertices if and only if $\displaystyle \cap_{j=0}^k U_{i_j}$ is non-empty. Moreover, the cover $\mathcal{C}$ is said to have multiplicity $\leq m+1$ if any point $p \in M$, it is contained in at most $(m+1)$ distinct members $U_i$ in $\mathcal{C}$. The diameter of the cover $\mathcal{C}$ is defined as
	\[\displaystyle \diam(\mathcal{C}) = \sup_{i \in I} \diam(U_i).\]
	A cover  $\mathcal{C}$ of $(M,g)$ is said to be uniformly bounded if $diam(\mathcal{C}) < \infty$.
\end{definition}

  We say an open cover $\mathcal{C}$ of $(M,g)$ has Lebesgue number at least $r$ if for $\forall p\in M$, the metric ball $B(p,r)$ of radius $r$ centered at $p$ is contained in some member of  the cover $\mathcal{C}$. For each member $U_i$ of $\mathcal C$, let $U_i^r \coloneqq N_r(U_i)  = \{x\in M: \dist(x, U_i) < r \}$.  Then $\mathcal{C}_r= \{ U_i^r\}_{i\in I} $ is a new open cover of $(M, g)$. 
Note that $\mathcal{C}_r$ satisfies
\[\diam(\mathcal{C}_r) \leq \diam(\mathcal{C}) + 2r,\]
and has Lebesgue number $ \geq r$. If we want to  emphasize the dependence on the Lebesgue number, we will denote by $\mathcal{N}_r$ the nerve of $(M,g)$ associated with the cover $\mathcal{C}_r$ of $(M,g)$. Clearly, all the above definitions apply to any metric space $(X,d)$. A metric space $(X,d)$ is said to  have Lebesgue dimension $\leq m$ if any open cover $\mathcal{C}$ of $(X,d)$ has a refinement $\mathcal{C}^\prime$ with multiplicity $m+1$. The Lebesgue dimension of $X$ is defined to be the minimum of such integer $m$. Now the notion of asymptotic dimension can be view as  a large scale version of the Lebesgue dimension.

\begin{definition} \label{def: finite_asymptotic_dimension}
	A metric space $(X,d)$ is said to have asymptotic dimension $m$ if $m$ is the minimal integer such that for any $R>0$, $X$ admits a uniformly bounded cover $\mathcal{C} = \{U_i\}_{i \in I}$ with $R$-multiplicity at most $m+1$. Here $R$-multiplicity means that  any ball $B(p,R) = \{x \in X :  d(x,p) < R\}$ intersects with at most $m+1$ members of  $\mathcal{C}$.
\end{definition}
Now we let $\mathcal{C} = \{U_i\}_{i \in I}$  be a uniformly bounded cover  of $(X, d)$ with $\diam(\mathcal{C})=R$ and $r$-multiplicity at most $m+1$ for any given $r > 0$. Let $\mathcal{N}_r$ be the nerve of $(X,d)$ associated with the cover $\mathcal{C}_r = \{U_i^r\}_{i \in I}$ of $(X,d)$ with Lebesgue number at least $r$. We equip $\mathcal{N}_r$ with the standard simplicial metric $d_{\mathcal{N}_r}$, which is equal to the standard $l^1$-metric $d_{l^1}$ on each simplex and, for two points in different simplicies, is defined to the distance of the shortest path connecting the two points. We define a map $f_r \colon  X \rightarrow \mathcal{N}_r$ by
\begin{equation}\label{eq: map_to_nerve}
	\displaystyle f_r(x) = \frac{\Sigma_{i \in I} d(x, X - U_i^r)U_i^r}{\Sigma_{i \in I} d(x, X - U_i^r)}.
\end{equation}
Now suppose $(X, d)$ has asymptotic dimension $m$. For each $r>0$, let $\mathcal N_r$ be the nerve complex associated with the cover $\mathcal C_r$ from above. Then the map $f_r$ from line \eqref{eq: map_to_nerve} satisfies the following properties: 
\begin{enumerate} 
	\item $f_r$ is $\frac{(m+1)^2}{r}$-Lipschitz;
	\item $f_r$ is uniformly cobounded. More precisely, for any given bounded subset $K \subset \mathcal{N}_r$ with diameter $\diam(K) \leq d$, then $\diam(f^{-1}(K)) \leq Rd+2(R+2r)$, where $R = \diam(\mathcal C)$.
\end{enumerate}
Indeed, consider the function $ f_{r,i}\colon X \to \mathbb R$ given by 
$$f_{r,i} (x)= \frac{d(x, X-U_i^r)}{\Sigma_{i \in I} d(x, X - U_i^r)}$$
Since the Lebesgue number of $\mathcal C_r$ is $\geq r$, we have 
$$\displaystyle \Sigma_{i \in I} d(\cdot, X - U_i^r) \geq r,  \text{ and the Lipschitz constant of } d(\cdot , X-U_i^r) \leq 1. $$ 
Hence, a direct computation shows  that 
$$\lip(f_{r,i}) \leq \frac{m+1}{r}.$$ Consequently,  $$\lip(f_r) \leq \frac{(m+1)^2}{r}.$$
Note that $\diam(f^{-1}(\sigma)) \leq 2(R+2r)$ for any simplex $\sigma \in \mathcal{N}_r$. Hence,  for any $K \subset \mathcal{N}_r$ with $\diam(K) \leq d$, we have  $$\diam(f^{-1}(K)) \leq Rd + 2(R+2r).$$

\begin{remark}\label{rmk:asysmptotic} The above discussion shows that if $(X, d)$ has asymptotic dimension $m$, then for any given $\varepsilon > 0$, there exists an $m$-dimensional locally finite simplicial complex $P^m$ and a Lipschitz map $f\colon M \rightarrow P^m $ such that  $\lip(f) \leq \varepsilon$. The latter in fact can be shown to be equivalent to the definition of finite asymptotic dimension. See the survey \cite{MR2725304} for more details. 
\end{remark}

\begin{definition}\label{def:asydimControl}
	Suppose that $\mathcal{D}\colon  \mathbb{R}_{\geq 0} \rightarrow  \mathbb{R}_{\geq 0}$ with $\mathcal{D}(s) \geq s$ and $(X,d)$ is a metric space. We say $X$  has asymptotic dimension $\leq m$ with diameter control $\mathcal{D}$ if for any $r >0$, $X$ admits a uniformly bounded cover $\mathcal{C}_r = \{U_i\}_{i\in I}$ with $r$-multiplicity at most $m+1$ and $\diam(\mathcal{C}_r) \leq \mathcal{D}(r)$. 
\end{definition}

 Next, we will construct a left homotopy inverse $g_r$ of the map $f_r\colon  X \rightarrow \mathcal{N}_r$ in $L^\infty(X)$, where $f_r$ is defined from  line \eqref{eq: map_to_nerve}. More precisely, we prove that 
\begin{lemma}\label{lemma: homotopy_equivalent}
	Suppose that $(X^n,g)$ is a complete, non-compact, Riemannian manifold with bounded geometry and finite asymptotic dimension. For each $r>0$, let $f_r\colon X\to \mathcal N_r$ be the map from line \eqref{eq: map_to_nerve}. Then there exists a continuous map  $g_r\colon \mathcal{N}_r \rightarrow  L^\infty(X)$ such that the following diagram commutes up to homotopy within $ N_{\mathcal D(r)}(X)$:
	$$\begin{tikzcd}
		X \arrow[r, "i", hook] \arrow[rd, "f_r"'] & L^\infty(X) \\
		& \mathcal N_r \arrow[u, "g_r"'] &         
	\end{tikzcd}$$
\end{lemma}
	\begin{proof}
		 For any vertex $U$ of $\mathcal{N}_r$, we associate $U$ with a specific point $p_U \in U \subset X  \subset L^\infty(X)$. This defines a map from the $0$-skeleton of $\mathcal{N}_r$ to $L^\infty(X)$, which we denote by $g_r^0$.

	For any point $y \in \mathcal{N}_r$ that lies in the simplex generated by ${U_0, U_1, \cdots , U_m}$, there are  non-negative numbers $a_i$ with $\sum_{i = 0}^{m} a_i = 1$ such that
		$$ y  = \sum_{i=0}^{m} a_i U_i.$$ 
We define the map $g_r\colon \mathcal{N}_r \rightarrow L^\infty(X)$ by setting
		$$g_r(y) = \sum_{i=1}^m a_i g_r^{0}(U_i).$$

Each point $p\in X$ is contained in finitely many open sets, say,  $V_0, \cdots, V_l$,  with $l\leq m$. Thus there exist non-negative $b_0, \cdots, b_l$ with $\sum_{i=0}^l b_i = 1$ such that
		 $$ f_r(p) = \sum_{i=1}^l b_i V_i. $$
It follows that 
		\[
		\dist(g_r(f_r(p)), i(p)) = \|\sum_{i=0}^l b_i g_r(V_i) - i(p)\| \leq \sum_{i=0}^lb_i\|g_r(V_i) - i(p)\| \leq \mathcal{D}(r).
		\]
	It follows that $\im(g_r\circ f_r) \subset N_{\mathcal{D}(r)}(X) \subset L^\infty(X)$. Clearly, $g_r \circ f_r$ is homotopy equivalent to $i$ by a linear homotopy  within $N_{\mathcal{D}(r)}(X) \subset L^\infty(X)$.
\end{proof}



\vspace{2mm}

 Finally, let us  review the Lipschitz controlled  quantitative $K$-groups from \cite{Wang:2021uj}. Let $(X,d)$ be a locally compact metric space, and $C_0(X)$ denote the algebra of continuous functions on $X$ that vanish at infinity, equipped with the supremum norm. An element $f$ in $M_k(C_0(X))$ is called $L$-Lipschitz if
\[\displaystyle \|f(x) - f(y)\| \leq L\cdot d(x,y), \quad \forall x ,y \in X,\]
where the norm on the left is the operator norm of matrices. Let us denote 
\begin{itemize}
	\item by $C_0(X)_L$  the collection of $L$-Lipschitz functions in $C_0(X)$.
	\item by $M_k(C_0(X))_L$ the collection of $L$-Lipschitz elements in $M_k(C_0(X))$.
\end{itemize}

 It is  clear that $\{C_0(X)_L\}_{L \geq 0}$ gives  a filtration of $C_0(X)$. If $X$ is non-compact, let $C_0(X)^+$ be the unitization of $C_0(X)$, that is, $C_0(X)^+$ is the algebra of continuous functions on the one-point compactification $X\cup\{\ast\}$  of $X$. And we define  
 \begin{equation}\label{eq:evaluation} \pi\colon  C_0(X)^+ \rightarrow \mathbb{C} 
 \end{equation}
to be the evaluation map at the point $\ast$.  

 For any given $L \geq 0$, we define  $P_k^L(C_0(X)^+)$ to be  the set of projections in $M_k(C_0(X)^+)_L$ with the inclusion
 \begin{equation}
	 P_k^L(C_0(X)^+) \rightarrow P_{k+1}^L(C_0(X)^+), p \mapsto \begin{pmatrix}
		 p&0\\
		 0&0
	 \end{pmatrix}.
 \end{equation}
 Similarly, we define $U_k^L(C_0(X)^+)$ to be the set of unitaries in $M_k(C_0(X)^+)_L$ with the inclusion
 \begin{equation}
	 U_k^L(C_0(X)^+) \rightarrow U_{k+1}^L(C_0(X)^+), u \mapsto \begin{pmatrix}
		 u&0\\
		 0&1
	 \end{pmatrix}.
 \end{equation}
We set 
 \begin{equation}
	P^L(C_0(X)^+) = \lim_{k\to \infty} P_k^L(C_0(X)^+) \textup{ and }  U^L(C_0(X)^+) = \lim_{k\to \infty} U_k^L(C_0(X)^+)
\end{equation}
under the natural inclusions from above.

\begin{definition}
The Lipschitz-controlled quantitative $K$-groups of $X$ are defined  as follows. 
\begin{enumerate}
	\item Let $(p,s) $ and $(q,t) \in P^L(C_0(X)^+) \times \mathbb{N}$.  We say $(p,s) \sim (q,t) $ if $p\oplus I_{j+t}$ and $q\oplus I_{j+s}$ are homotopic in $P^{2L}(C_0(X)^+)$ for some $j \in \mathbb{N}$, where $I_{j+t}$ is the $(j+t)\times (j+t)$ identity matrix.   We define 
	\begin{equation*}
		K_0^L(C_0(X)) : = \{(p,\ell) \in P^L(C_0(X)^+): \rank(\pi(p)) = \ell\}/ \sim,
	\end{equation*}
	where $\pi$ is the evalutation map from line \eqref{eq:evaluation}. 
	\item Let $u$ and $v \in U^L(C_0(X)^+)$. We say $u \sim v$ if $u$ and $v$ are homotopic in $U^{2L}(C_0(X)^+)$. We define 
    \begin{equation*}
		K_1^L(C_0(X)) : = U^L(C_0(X)^+)/ \sim.
	\end{equation*}
\end{enumerate}

\end{definition}
 It is not difficult to see that $K_*^L(C_0(X)^+)$ are abelian groups for any $L \geq 0$. In fact, the quantitative $K$-theory applies to a much broader class of $C^\ast$-algebras. We would like to refer the reader to \cite{Oyono-OyonoYu} for a  systematic approach to the quantitative $K$-theory for more general $C^\ast$-algebras. 

Recall that we used $l^1$-metric on the  nerve space  $\mathcal{N}_r$ above. In fact, there is quite a bit of flexibility when it comes to the choice of such a metric. It turns out that, in order to improve the various contants that will appear in the proof of our main theorem, let us consider a $l^2$-version of simplicial metric on $\mathcal N_r$ as follows. For each $n>0$, let $\overbar\Delta_n$ be the regular $n$-simplex in $\R^{n}$ centered at the origin. More precisely, the vertices $\{v_n^j\}_{0\leq j\leq n}$ of $\overbar\Delta_n$ are defined inductively by:
$$\begin{cases}
	v_1^0=1,~v_1^1=-1,\\
	\displaystyle v_n^n=0\oplus e_{n+1},~v_n^j=\sqrt{1-\frac{1}{n^2}} v_{n-1}^j\oplus - \frac 1 n e_n,
\end{cases}$$
with $e_n$ the unit vector of the $n$-th direction of $\R^n$.

Let $\widetilde\Delta_n$ be the simplex in the unit ball of $\R^n$, whose vertices are $\{v_n^j\}_{0\leq j\leq n}$ and faces are the spherical simplices generated by the vertices. Equip $\widetilde\Delta_n$ with the Euclidean metric $\widetilde d$. Now for the simplicial complex $\mathcal N_r$, we equip each top-dimensional simplex with the metric $\widetilde d$ and define the distance between any two points in two different simplices to the distance of the short path connecting the two points. We denote the resulting metric on $\mathcal N_r$ by $\widetilde d$, since no confusion is likely arise.  Note that, with respect to the new metric $\widetilde d$ on  $\widetilde\Delta_n$, every Lipschitz function $f\colon \partial\widetilde\Delta_n\to M_k(\mathbb C)$ with Lipschitz function $L$ extends to a Lipschitz function on $\widetilde \Delta_n$ by 
$$\widetilde f(x)=\begin{cases}
	f(\frac{x}{|x|})(2|x|-1),&|x|\geq 1/2,\\
	0,&|x|\leq 1/2,
\end{cases}$$
where  $|x|$ is the distance of $x$ to the origin. A key observation is that the Lipschitz constant of $\widetilde f$ is  at most $2L+2\|f\|$ which is  \emph{independent} of the dimension of $\widetilde \Delta_n$. 

Now let $f_r$ be the map from line \eqref{eq: map_to_nerve}. Note that the radial projection from $(\Delta_n,d_{l^1})$ to $ (\widetilde\Delta_n,\widetilde d)$ has Lipschitz constant at most $n$. It follows that with respect to the new metric $\widetilde d$ on $\mathcal N_r$, we have 
\begin{equation}\label{eq: lip_cube}
	\lip(f_r) \leq \frac{(m+1)^3}{r}.
\end{equation}

A key result of the Lipschitz controlled $K$-theory is the following theorem from {\cite{Wang:2021uj}},  which states that $K$-theory classes of finite dimensional simplicial complexes have Lipschitz controlled representatives.  

\begin{theorem}[{\cite{Wang:2021uj}}]\label{thm: K_rep_theorem}
	Let  $(\mathcal N,\widetilde d)$ be a locally compact $m$-dimensional simplicial complex equipped with the metric $\widetilde d$ as above. Then  there exists a constant $L_m$ such that every class in $K_  *(C_0(\mathcal N))$ admits an $L_m$-Lipschitz representative in $K^{L_m}_*(C_0(\mathcal N))$. Here $L_m\leq C_1e^{C_2m}$ for some universal positive constants $C_1$ and $C_2$ independent of $\mathcal N$. Moreover, if $\alpha \in K_  *(C_0(\mathcal N))$ can be represented by an element that is constant outside a compact $K\subset \mathcal N$, then one can choose an $L_m$-Lipschitz representative of $\alpha$ that is constant outside the $1$-neighborhood of $K$.
\end{theorem}

For the convenience of the reader, let us briefly sketch the proof of Theorem \ref{thm: K_rep_theorem}. The complete details can be found in    \cite{Wang:2021uj}. For simplicity, let us $\mathcal N$ is compact. Let $\mathcal N_1$ be a small neighborhood of the $(m-1)$-skeleton of $\mathcal N$, and $\mathcal N_2$ be the interior of all $m$-simplices of $\mathcal N$. Then $\mathcal N_1$ is Lipschitz homotopic to the $(m-1)$-skeleton of $\mathcal N$, $\mathcal N_2$ is Lipschitz homotopic to a disjoint union of points, and $\mathcal N_1\cap\mathcal N_2$ is Lipschitz homotopic to a disjoint union of $(m-1)$-spheres. 

Theorem \ref{thm: K_rep_theorem} is proved by induction. First, Theorem \ref{thm: K_rep_theorem}  obviously holds for $\mathcal N_2$ with zero Lipschitz constant. Note that $\mathcal N_1\cap\mathcal N_2$ and $\mathcal N_1$ are essentially of dimension $(m-1)$. By induction, assume that Theorem \ref{thm: K_rep_theorem} holds for $\mathcal N_1\cap\mathcal N_2$ and $\mathcal N_1$ with Lipschitz constants $L_{m-1}$.
Let 
$$\pi_i\colon C(\mathcal N_i)\to C(\mathcal N_1\cap \mathcal N_2),~i=1,2$$
be the restriction map between the algebras of continuous functions, which is surjective for $i=1,2$. Furthermore, $\pi_i$ is also surjective with respect to the Lipschitz filtration, in the sense that every Lipschitz function $f\in M_k(C(\mathcal N_1\cap \mathcal N_2))_L$ extends to a Lipschitz function in $M_k(C(\mathcal N_i))$ with Lipschitz constant at most $2L+2\|f\|$. Now the partition $\mathcal N=\mathcal N_1\cup \mathcal N_2$ induces a Mayer--Vietoris six-term exact sequence  of $K$-theory groups $K^*(\mathcal N)$, $K^*(\mathcal N_1)\oplus K^*(\mathcal N_2)$, and $K^*(\mathcal N_1\cap\mathcal N_2)$, which is  asymptotically exact if we additionally keep track of the Lipschitz filtration; see \cite[Definition 4.5 \& Proposition 4.10]{Wang:2021uj}. Since the Lipschitz constant of the extension (of Lipschitz functions) is independent of $m$, there exist universal constants $A_1,A_2>0$ such that the desired constant $L_m$ is at most 
\begin{equation}\label{eq:A1A2}
	L_m\leq A_1\cdot L_{m-1}+A_2.
\end{equation}
See \cite[Appendix]{Wang:2021uj}. Now Theorem \ref{thm: K_rep_theorem} follows by the induction on the dimension of the skeleton, which ends in  $m$ steps. This shows that 
\begin{equation*}
	L_m\leq C_1 e^{C_2 m}
\end{equation*}
for some universal constants $C_1,C_2>0$.


\begin{remark}
	By tracing through the proofs of the five lemma \cite[Lemma A.19]{Wang:2021uj} and the asymptotic exact sequences \cite[Proposition A.26 \& Theorem A.35]{Wang:2021uj}, one sees that the constants $A_1$ and $A_2$ in the above discussion are at most $10^{20}$. As a result, we have $C_1\leq 10^{20}$ and $C_2\leq 50$, where $C_1$ and $C_2$ are  the constants in Theorem \ref{thm: K_rep_theorem}. The bounds $10^{20}$ and $50$ are not optimal. It is an interesting question what are the optimal constants. 
\end{remark}
\section{Proof of Theorem \ref{thm:main}}\label{sec:main}

In this section, we prove the main theorem (Theorem \ref{thm:main}).
\subsection{Index pairing}
To prepare the proof of Theorem \ref{thm:main}, let us first review some basics of the index pairing.  We shall mainly focus on the even dimensional case. The odd dimensional case is similar. Let $M$ be a complete spin manifold of even dimension. The Dirac operator on $M$ defines a $K$-homology class $[D]$ in $K_0(M)$, and there is a pairing map
\begin{equation}\label{eq:pairing}
	K^0(M)=K_0(C_0(M))\to K_0(\mathcal K)\cong\Z,
\end{equation}
given by 
\[\alpha\mapsto\langle [D],\alpha\rangle,  \]
where $C_0(M)$ is the algebra of continuous functions on $M$ vanishing at infinity, and $\mathcal K$ is the algebra of compact operators on a Hilbert space. Let us describe an explicit formula for $\langle [D],\alpha\rangle$, cf.  \cite{Willett_Yu_Book_higher_index_theory}. 


Let us  first review the difference construction of $K$-theory elements from \cite{KasparovYuconvex}. 
\begin{lemma}\label{lemma:double}
	If $A$ is a $C^*$-algebra and $I$ is an ideal of $A$, and we define
	$$D_A(I)=\{(a,a')\in A\oplus A: a-a'\in I \},$$
	then there is a natural surjective homomorphism from $K_*(D_A(I))$ to $K_*(I)$. In other words, each $K$-theory class of the ideal $I$ can be expressed by $[\alpha]-[\beta]$ for some  $\alpha,\beta\in M_k(A^+)$ with $\alpha-\beta\in M_k(I)$.
\end{lemma}
\begin{proof}
	Consider the following short exact sequence
	$$0\longrightarrow I\longrightarrow D_A(I)\longrightarrow A\longrightarrow 0,$$
	where the map $I\to D_A(I)$ is given by $a\mapsto (a,0)$. This exact sequence splits since there exists a section $A\to D_A(I)$ given by $a\mapsto (a,a)$. Thus, $K_*(D_A(I))\cong K_*(I)\oplus K_*(A)$,  and the required surjective homomorphism from $K_*(D_A(I))$ to $K_*(I)$ is given by the projection from $K_*(D_A(I))$ to the first component.
\end{proof}	

Now suppose $[\alpha]-[\beta]$ is a $K$-theory class in  $K_0(I)$ from Lemma $\ref{lemma:double}$. We shall review the difference construction \cite{KasparovYuconvex}, which provides an explicit way to construct idempotents of the unitization $I^+$ that  represent  the same class as $[\alpha]-[\beta]$ in $K_0(I)$.  Without loss of generality, we may assume $\alpha,\beta$ are idempotents in $A^+$ such that $\alpha-\beta\in I$. Set
$$Z(\beta)=\begin{pmatrix}
	\beta &0&1-\beta &0\\ 1-\beta &0&0&\beta\\0&0&\beta&1-\beta\\0&1&0&0
\end{pmatrix}\in M_4(A^+).$$
and it is invertible with inverse
$$Z(\beta)^{-1}=\begin{pmatrix}
	\beta &1-\beta &0&0\\
	0&0&0&1\\
	1-\beta &0&\beta&0\\0&\beta&1-\beta&0
\end{pmatrix}\in M_4(A^+).$$
We define
\begin{equation}\label{eq:diffe}
	\begin{split}
		d(\alpha,\beta) \coloneqq  & Z(\beta)^{-1}\begin{pmatrix}
			\alpha&&&\\
			&1-\beta&&\\
			&&0&\\
			&&&0
		\end{pmatrix}Z(\beta)\\=&\begin{pmatrix}
			1+\beta(\alpha-\beta)\beta&0&\beta(\alpha-\beta)&0\\
			0&0&0&0\\
			(\alpha-\beta)\alpha\beta&0&(1-\beta)(\alpha-\beta)(1-\beta)&0\\
			0&0&0&0
		\end{pmatrix}\in M_4(I^+)
	\end{split}
\end{equation}
and 
$$
e_{1,3} \coloneqq \begin{psmallmatrix}
	1&&&\\
	&0&&\\
	&&0&\\
	&&&0
\end{psmallmatrix}.
$$
Note that both $d(\alpha,\beta)$ and $e_{1, 3}$ are idempotents in
$M_4(I^+)$, and  $$[\alpha]-[\beta] = [d(\alpha,\beta)]-[e_{1,3}] \in K_0(I).$$

Now we return to the construction of the index pairing. Let $M$ be an even-dimensional complete spin manifold and $D$ the Dirac operator acting on its spinor bundle $S$. The spinor bundle is naturally $\Z_2$-graded and the Dirac operator is an odd operator of the form:  
$$D=\begin{pmatrix}
	0&D_+\\D_-&0
\end{pmatrix}.$$

Recall that a continuous function $\chi\colon \mathbb R \to [-1, 1]$ is called \emph{normalizing function} if $\chi$ is a non-decreasing odd function  such that 
$\displaystyle \lim_{x\to \pm \infty} \chi(x) = \pm 1$. To be specific, we consider the following the smooth function $\chi$ given by
$$\chi(x)=\frac{2}{\pi}\int_0^x\frac{1-\cos y}{y^2}dy -1.$$
It is easy to verify that $\chi$ is a normalizing function. Moreover, the Fourier transform $\widehat\chi$ has compact support.  Note that, for any $t>0$, the operator $\chi(t^{-1}D)$ is a self-adjoint odd operator of the form: 
\begin{equation}\label{eq:chi}
	\chi(t^{-1} D)=\begin{pmatrix}
		0&U_{t,D}\\V_{t,D}&0
	\end{pmatrix},
\end{equation}
where $U_{t,D}$ and $V_{t,D}$ are bounded operators acting on the Hilbert space $L^2(M,S)$ of $L^2$-sections of the spinor bundle $S$. Since $\widehat \chi$ has compact support, the operator   $\chi(t^{-1}D)$ has finite propagation, and its propagation goes to zero as $t\to\infty$. It follows that
\begin{equation}\label{eq:propa}
	\|[\chi(t^{-1}D),f]\|\to 0,\text{ as }t\to\infty
\end{equation}
for any $f\in C_0(M)$, where $[\chi(t^{-1}D),f] = \chi(t^{-1}D)\cdot f - f\cdot \chi(t^{-1}D)$ is the commutator of $\chi(t^{-1}D)$ and $f$.
Moreover, since $\chi^2-1\in C_0(\R)$, the operators
\begin{equation}\label{eq:localCompact1}
	f\cdot \left(\chi(t^{-1} D)^2-1\right)=\begin{pmatrix}
	f\cdot (U_{t,D}V_{t,D}-1)&0\\0&f\cdot (V_{t,D}U_{t,D}-1)
\end{pmatrix}
\end{equation}
\begin{equation}\label{eq:localCompact2}
	\left(\chi(t^{-1} D)^2-1\right)\cdot f=\begin{pmatrix}
	(U_{t,D}V_{t,D}-1)\cdot f&0\\0& (V_{t,D}U_{t,D}-1)\cdot f
\end{pmatrix}
\end{equation}
are compact operators for any $f\in C_0(M)$. 

Set
$$W_{t,D}=\begin{pmatrix}
	1& U_{t,D}\\0&1
\end{pmatrix}
\begin{pmatrix}
	1&0\\-V_{t,D}&1
\end{pmatrix}
\begin{pmatrix}
	1&U_{t,D}\\0&1
\end{pmatrix}
\begin{pmatrix}
	0&-1\\1&0
\end{pmatrix},\ e_{1,1}=\begin{pmatrix}
	1&0\\0&0
\end{pmatrix}
$$
and
\begin{equation}\label{eq:PtD}
	\begin{split}
		P_{t,D} = & W_{t,D}e_{1,1}W_{t,D}^{-1}\\
		= & \begin{pmatrix}
			1-(1-U_{t,D}V_{t,D})^2 & (2-U_{t,D}V_{t,D})U_{t,D}(1-V_{t,D}U_{t,D})\\
			V_{t,D}(1-U_{t,D}V_{t,D})&(1-V_{t,D}U_{t,D})^2
		\end{pmatrix}.
	\end{split}
\end{equation}
The formal difference of $P_{t,D}$ and $e_{1,1}$ gives an explicit representative of the $K$-homology class $[D]$, cf. \cite[Remark 8.3.15]{Willett_Yu_Book_higher_index_theory}.

Now let $[p]-[q] \in K_0(C_0(M))$, where $p,q$ are projections in $M_k(C_0(M)^+)$ such that $p-q \in M_k(C_0(M))$ for some $k \in \mathbb{N}$. 
To simplify notation, let   $P_t = P_{t, D_k}$, where $D_k=D\otimes I_k$ for some $k \in \mathbb{N}$. It follows from line \eqref{eq:propa} that 
$$\|[P_t,p]\|\to 0  \textup{ and } \|[P_t,q]\|\to 0, \text{ as }t\to \infty.$$
Consequently,
\begin{equation}\label{eq:asymIdem1}
	\|(P_{t}\cdot p)^2-P_{t}\cdot p\| \to 0 \textup{ and }  \|(P_{t}\cdot q)^2-P_{t}\cdot q\| \to 0,  \textup{ as }   t\to \infty. 
\end{equation}

Now we apply the difference construction in line \eqref{eq:diffe} to define $$a_{t,p,q} = d(P_t\cdot p,P_t\cdot q)  \text{ and } b_{p,q} = d(e_{1,1}\otimes p,e_{1,1}\otimes q),$$ where $e_{1,1}=\begin{psmallmatrix}1&0\\0&0\end{psmallmatrix}$. We see that $b_{p,q}^2=b_{p,q}$, and $a_{t,p,q}$ is an  almost idempotent in the sense that
\begin{equation}\label{eq:asymIdem2}
	\|a_{t,p,q}^2-a_{t,p,q}\|\to 0,  \textup{ as }   t\to \infty.
\end{equation}
We further consider the difference construction of $a_{t,p,q}$ and $b_{p,q}$ by setting
$$d_{t,p,q}\coloneqq d(a_{t,p,q}, b_{p,q})\text{  and }e\coloneqq e_{1,3}\otimes I_k\otimes I_4,$$
where $I_{k}$ denotes the $(k\times k)$ identity matrix  and 
\[ e_{1,3} = \begin{psmallmatrix}
	1 & & & \\ 
	& 0 & & \\
	& & 0 & \\
	& & & 0
\end{psmallmatrix}. \] 

By the construction in \eqref{eq:diffe} and \eqref{eq:PtD}, each entry of the matrix $(d_{t,p,q}-e)$ is a (noncommutative) polynomial of $U_{t,D}$, $V_{t,D}$, $p$ and $q$. Each monomial term of the polynomials  contains the factors $1-U_{t,D}V_{t,D}$ (or $1-V_{t,D}U_{t,D}$) and $p-q$.
It follows from line \eqref{eq:localCompact1} and \eqref{eq:localCompact2} that
$$d_{t,p,q}\in M_{32k}(\mathcal K^+),~\forall t>0,$$
and
$$ d _{t,p,q} - e = d(a_{t,p,q}, b_{p,q})-e_{1,3}\otimes I_{2k}\otimes I_4\in M_{32k}(\cK),~\forall t>0.$$

By line \eqref{eq:asymIdem1} and \eqref{eq:asymIdem2}, $d_{t,p,q}$ is an almost idempotent in the sense that
\begin{equation}\label{eq: decay_to_zero}
	\|d_{t,p,q}^2-d_{t,p,q}\|\to 0,  \textup{ as }   t\to \infty
\end{equation}
More precisely, we derive the following quantitative estimate of $\|d_{t,p,q}^2-d_{t,p,q}\|$. 
\begin{lemma}[{ \cite[Lemma 3.4]{Wang:2021uj}}]\label{lemma:key1} With the same notation as above, if $p,q \in M_k(C_0(M)^+)$ be Lipschitz functions with Lipschitz constant $L$, then there exists a universal positive constant $c_1$ such that
	$$\|d^2_{t,p,q} - d_{t,p,q}\| \leq \frac{c_1 L}{t}.$$
\end{lemma}

In particular, if $t> 4c_1L$, Lemma \ref{lemma:key1} implies that
\[\|d^2_{t,p,q} - d_{t,p,q}\| < \frac{1}{4}.\]
Now we apply the standard holomorphic functional calculus to obtain a genuine idempotent from $ d_{t,p,q}$. More precisely, consider the  holomorphic function
\[ \displaystyle \Theta(z) = \left\{
\begin{array}{ll}
	0 ,&  \text{if Re}(z) < \frac{1}{2}; \vspace{.5cm}\\
	1 ,&  \text{if Re}(z) \geq \frac{1}{2}.
\end{array}
\right.\]
We define
\begin{equation}
	\Theta(d_{t,p,q}) = \frac{1}{2\pi i} \int_{\gamma}(d_{t,p,q} - \xi)^{-1} d\xi, 
\end{equation}
where $\gamma = \{z \in \mathbb{C} : |z-1| =\frac{1}{2}\}$. and then $\Theta(d_{t,p,q})$ is an idempotent. 
\begin{definition}[{cf. \cite[\S 9.1]{Willett_Yu_Book_higher_index_theory}}]\label{eq: theta}
	The index pairing \eqref{eq:pairing} of the Dirac operator $D$ and $[p]-[q]\in K_0(M)$ is explicitly given by
	 $$\langle [D], [p]-[q] \rangle = [\Theta(d_{t, p,q})]-[e]\in K_0(\mathcal K)=\Z$$ for any  $t> 4c_1L$.
\end{definition}


Now suppose that Riemannian manifod $(M^n,g)$ has uniformly positive scalar curvature $\Sc_g \geq \sigma^2 >0$. It follows from the Lichnerowicz formula
$$D^2=\nabla^*\nabla+\frac{\Sc_g}{4}$$
that the Dirac operator $D$ of $M$  has a spectral gap at zero. In this case, we see that 
$$\|\chi(t^{-1}D)^2-1\|\to 0,\text{ as }t\to 0.$$ 
Moreover, we have the following lemma from  \cite{Wang:2021uj}.
\begin{lemma}[{\cite[Lemma 3.4]{Wang:2021uj}}] \label{lemma:key2} 
	Suppose that $(M^n, g)$ is a complete, non-compact Riemannian manifold. If $p-q$ is supported on a compact set in $M$ and $\Sc_g \geq \sigma^2 >0$ on $M$, then there exists a universal constant constant $c_2$ such that
		$$ \displaystyle \|d_{t,p,q}-e\|\leq \frac{c_2 t}{\sigma}$$
		for all $t>0$.
\end{lemma}


With the above preparation, we have the following key estimate which shows  if a complete Riemannian manifold $(M^n, g)$ has uniformly positive scalar curvature, then the pairing of its Dirac operator with a bundle with sufficiently small Lipschitz constant has zero index. 

\begin{proposition}[{\cite[Proposition 3.5]{Wang:2021uj}}] \label{thm: Scalar_curvatyre_Lip}
	Suppose that $(M^{n},g)$ is a complete, noncompact, spin Riemannian manifold with uniformly positive scalar curvature $\Sc_g \geq \sigma^2 > 0$. Let $D$ be the Dirac operator on $M$ and $[p]-[q]\in K_0^L(M)$ with $p,q$ projections in $M_k(C_0(M)^+)$. If 
	 \begin{equation}
	L < \frac{\sigma}{16c_1c_2},
\end{equation}
	then the index pairing $\langle [D], [p]-[q] \rangle = 0$.
\end{proposition}
\begin{proof}
	Let us retain the same notation from  Definition \eqref{eq: theta}. We have
	$$\Theta(d_{t,p,q}) = \frac{1}{2\pi i} \int_{\gamma}(d_{t,p,q} - \xi)^{-1} d\xi,$$
	with $\displaystyle \gamma = \{z \in \mathbb{Z} : |z-1| =\frac{1}{2}\}$ for $t> 4c_1L$.
	
Let $t_0 = 4c_1L$, we have
\begin{equation}\label{eq: less_than_1/4}
	 		\|d_{t_0,p,q}-e\|  \leq  \frac{c_2t_0}{\sigma}
	\leq  \frac{c_2 \cdot 4c_1 L}{\sigma} < \frac{1}{4}.
\end{equation}
Moreover, we have 
	 \begin{align}
	 	\Theta(d_{t_0,p,q})-e=&
	 	\frac{1}{2\pi i}\int_\gamma ((d_{t_0,p,q}-\xi)^{-1}-(e-\xi)^{-1})d\xi \notag \\
	 	=&\frac{1}{2\pi i}\int_\gamma (d_{t_0,p,q}-\xi)^{-1}(e-d_{t_0,p,q})(e-\xi)^{-1}d\xi. \label{eq:holo}
	 \end{align}
	 Note that
	 $$
	 (d_{t_0,p,q}-\xi)^{-1}=(e-\xi)^{-1}(1+(d_{t_0,p,q}-e)(e-\xi)^{-1})^{-1}.
	 $$ 
	 For any $ \xi \in \gamma = \{z : |z-1| = \frac{1}{2}\}$, we have
	 \begin{equation}\label{eq: estimate_on_gamma}
	 	\|(e-\xi)^{-1}\|\leqslant 2 \textup{ and } \|(d_{t_0,p,q}-\xi)^{-1}\|<4.
	 \end{equation}
	 Applying the inequalities  (\ref{eq: less_than_1/4}) and (\ref{eq: estimate_on_gamma}) to the integral $\eqref{eq:holo}$,  we have 
	 \begin{equation}\label{eq:idemestimate}
	 	\|\Theta(d_{t_0,p,q})-e\| <1.
	 \end{equation}
	 Recall that if two idempotents $f_1$ and $f_2$ in a $C^\ast$-algebra  satisfy  the inequality 
	 \[ \|f_1 - f_2\| < \frac{1}{\|2f_1 - 1\|}, \]
	 then $f_1$ is equivalent to $f_2$; see \cite[Proposition  4.3.2]{Blackadar}. In our case, since $e$ is a projection, we have 
	 \[ \|2e-1\| = 1. \]	In particular, the inequality in line $\eqref{eq:idemestimate}$ implies that 
	 \[ \|\Theta(d_{t_0,p,q})-e\|<\frac{1}{\|2e -1\|}. \]
	 It follows that $\Theta(d_{t_0,p,q})$ is equivalent to $e$. 
	 Thus $[\Theta(d_{t_0,p,q})]-[e] = 0 \in K_0(\cK)$. 
	 This implies $\langle [D], [p]-[q] \rangle = 0$.
	 \end{proof}
	 
	 
\begin{remark}\label{remark:Scalar_curvatyre_Lip}
	Equivalently, we may state Proposition \ref{thm: Scalar_curvatyre_Lip} as follows: if the index pairing $\langle [D], [p]-[q] \rangle$ is non-zero, then 
		 \begin{equation}
		L \geq \frac{\sigma}{16c_1c_2}.
	\end{equation}
\end{remark}

\subsection{Proof of the main theorem}
Now let us prove Theorem \ref{thm:main}.
\begin{proof}[Proof of Theorem \ref{thm:main}]
Let us first prove the even dimensional case.  Let $(M,g)$ be a complete, non-compact spin Riemannian manifold of even dimension. For a given $r_0 > 0$, suppose $\mathcal{D}(r_0) < \fillrad_{\mathbb{Q}}(M) $. By the definition of filling radius in Definition \ref{def: filling_radius}, $[M]$ is  non-zero in $H^{\lf}_n(N_{\mathcal{D}(r_0)}(M; \mathbb{Q}))$. Let $f_{r_0}$ be the map constructed in line \eqref{eq: map_to_nerve} and  $g_{r_0}$ the corresponding map constructed in Lemma \ref{lemma: homotopy_equivalent}. In particular,  we have 
$$g_{r_0}f_{r_0} \textup{ is homotopic to } i\colon M\to N_{\mathcal D(r_0)}(M).$$
Consequently, $$[M] = i_{*}([M]) = (g_{r_0})_*\left((f_{r_0})_*([M])\right).$$
 Since $[M]$ is  non-zero in $H^{\lf}_n(N_{\mathcal{D}(r_0)}(M; \mathbb{Q}))$, it follows that $(f_{r_0})_*([M])$ is non-zero in $H^{\lf}_n(\mathcal N_{r_0}; \mathbb{Q})$.

The Dirac operator $D$ of $M$ defines a $K$-homology class $[D]$ in $K_0(M)$. 
Under the natural isomorphism\footnote{This isomorphism is induced by the Chern character map. } $K_0(M)\otimes \mathbb Q\cong \oplus_{even}H^{\lf}_*(M,\mathbb Q)$, the top degree component of $[D]$ in $H^{\lf}_n(M,\mathbb Q)$ coincides with the fundamental class $[M]\in H^{\lf}_n(M,\mathbb Q)$. For example, see the Poincar\'{e} duality theorem in \cite[Section 9.6]{Willett_Yu_Book_higher_index_theory}. In particular,  we conclude that $(f_{r_0})_*([D])$ is a nonzero element in  $K_0(\mathcal N_{r_0})$, since its top degree component in $H^{\lf}_n(\mathcal N_r,\mathbb Q)$ is equal to the non-zero element $(f_{r_0})_*([M])$ in $H^{\lf}_n(\mathcal N_r,\mathbb Q)$. Since the index pairing $K_0(\mathcal N_{r_0})\otimes K^0(\mathcal N_{r_0}) \to \mathbb Z$ is (rationally) nondegenerate,   there exists a  $K$-theory class $\beta\in K^0(\mathcal N_{r_0})$ such that
\[\langle (f_{r_0})_*([D]), \beta \rangle \neq 0.\]

By  Theorem \ref{thm: K_rep_theorem}, there exist $L_m$-Lipschitz projections $p, q\in M_k(C_0(\mathcal N_{r_0})^+)$ for some $k\in\mathbb N$ such that $$ \beta = [p]- [q]\in K^0(\mathcal N_{r_0}).$$ 	It follows that
$$0 \neq \langle (f_{r_0})_*([D]), \beta \rangle = \langle [D], f_{r_0}^*([p]-[q]) \rangle.$$
	
Recall the  inequality \eqref{eq: lip_cube} (see also the discussion after Definition \ref{def: finite_asymptotic_dimension} for more details):
	\begin{equation} \label{eq: Lip_estimate}
		\displaystyle	\lip(f_{r_0}) \leq \frac{(m+1)^3}{r_0}.
	\end{equation}
It follows that 
 $$\lip(f_{r_0}^*(p))\leq \frac{(m+1)^3L_m}{r_0} \text{ and } \lip(f_{r_0}^*(q))\leq \frac{(m+1)^3L_m}{r_0}.$$ 
Now by applying Theorem \ref{thm: Scalar_curvatyre_Lip}, we have 
	$$ \displaystyle \frac{(m+1)^3L_m}{r_0} \geq \frac{{\sigma}}{16c_1c_2}.$$
Equivalently, we have 
	$$ r_0 \leq \frac{16c_1c_2 (m+1)^3L_m}{{\sigma}}.$$
	To summarize, we have shown that
	\[ \textup{if } r_0 > \frac{16c_1c_2 (m+1)^3L_m}{{\sigma}}, \textup{ then } \mathcal D(r_0) \geq \fillrad_{\mathbb{Q}}(M). \]
	In other words, we have 
	 $$ \displaystyle \fillrad_{\mathbb{Q}}(M) \leq \mathcal{D} \left(\frac{16c_1c_2 (m+1)^3L_m}{{\sigma}}\right).$$
	This proves the theorem for the even dimensional case.

	Now assume that $M$ is odd dimensional and then consider the direct product space $(M \times \mathbb{R}, g+dt^2)$. Since the scalar curvature of $g$ is uniformly bounded below by $\sigma^2$, it follows that the scalar curvature of $g+dt^2$ is also bounded below by  $\sigma^2$. Moreover, the asymptotic dimension of $M\times \mathbb R$ equals $m+1$, where $m$  is the asymptotic dimension of $M$. This reduces the odd dimensional case to the even dimensional case. In particular, the same argument above shows that 	$$\fillrad_{\mathbb{Q}}(M \times \mathbb{R}) \leq \mathcal{D} \left(\frac{16c_1c_2 (m+2)^3L_{m+1}}{{\sigma}}\right).$$
	Now by Lemma  \ref{remark: fillrad_of_direct_roduct}, we conclude that
	$$\fillrad_{\mathbb{Q}}(M) = \fillrad_{\mathbb{Q}}(M \times \mathbb{R}) \leq \mathcal{D}\left(\frac{16c_1c_2 (m+2)^3L_{m+1}}{{\sigma}}\right).$$
	This finishes the proof of Theorem \ref{thm:main}.
\end{proof}

While the proof of Theorem \ref{thm:main} does not rely on the (coarse) Baum--Connes conjecture, the techniques used in the proof are  closely related to those developed for the (coarse) Baum--Connes conjecture \cite{Alain_ndg, Kasparov_equi_kk,bc_geometric_k,ConnesNCG,guentner2022dynamical,willett2023uct}. In fact, a counterexample to Conjecture \ref{conj: gromov_filling_radius_conjecture} could potentially lead to counterexamples to the Baum-Connes conjecture. 


\appendix

\section{Metric invariants of Riemannian manifolds} \label{appdendix: metric_invariant}

In this appendix, we review a list of metric invariants of Riemannian manifolds and their relations with each other. We also briefly explain how these metric invariants are connected with some of the recent developments on metric invariants and positive scalar curvature problems. We refer the reader to  \cite{gromovfilling,gromovmetricstructure,MR2373013, Gromovinequalities2018, gromov2023_injectivity_radius, Gromovadozen,MR3622232,MR2827817,MR2753599} for more systematic studies of these metric invariants.

\subsection{Metric invariants}

We first recall the definition of Uryson width due to Gromov \cite{gromovmetricstructure}. 

\begin{definition} \label{def: width}
	Suppose that $(M^n,g)$ is a  Riemannian manifold, the  $k$-th Uryson width of $(M,g)$ is defined as
	\begin{equation}
		\width_{k}(M) = \inf_{f, P^k} \ \sup_{p \in P^k}\{\text{diam}(f^{-1}(p)) \mid  f\colon  M \rightarrow P^k\}
	\end{equation}
	where $P^k$ is a locally finite $k$-dimensional polyhedron and $f$ is a continuous map from $M$ to $P^k$.  We then define the macroscopic dimension of $M$ to be 
	\[\dim_{\macro}(M) = \inf_{k} \{k: \width_k(M) < \infty\}.\]
\end{definition}

The Uryson width can be viewed as a generalization of the notion of diameter of Riemannian manifolds. Intuitively speaking,  the $k$-th Uryson width of $M$ quantifies the shortest ``distance'' from the $n$-dimensional Riemannian manifold $(M^n, g)$ to a locally finite simplical complex  $P^k$ of dimension $k$. In particular,  $\width_0(M) = \diam(M)$, where $ \displaystyle \diam(M) = \sup_{p, q \in M} {d_g(p,q)}$ is the diameter of $(M, g)$.

\begin{lemma} \label{lemma: width_sequence}
If  $(M^n,g)$ is a complete Riemannian manifold, then
	\begin{enumerate}[label=$(\arabic*)$]
		\item \label{eq: width_monotone_sequence} $\width_{k+1}(M) \leq \width_{k}(M)$ for any $k \geq 0$; 
		\item $\width_0(M) = \diam(M)$ and $\width_k(M) =0$ for $k \geq n$.
	\end{enumerate}
\end{lemma}
\begin{proof}
	First, it follows from definition that $\width_{k+1}(M) \leq \width_{k}(M)$ for any $k$. For $\width_n(M)$, let us fix a locally finite triangulation of $M$ and denote the resulting simplicial complex by $P$. Then it is clear that the identity map from $M$ to $P$ implies that $\width_n(M) =0$. It follows by part (1) that $\width_{k}(M) = 0$ for all $k \geq n$. For $\width_0(M)$, we consider the constant map from $M$ to a single point, which implies that $\width_0(M) = \diam(M)$. 
\end{proof}

\begin{lemma} \label{thm: uryson_with_equivalent}
	Suppose that $(M^n,g)$ is a complete Riemannian manifold. For any given $W > 0$, the following statements are equivalent:
	\begin{enumerate}[label=$(\arabic*)$, ref=\arabic*]
		\item \label{item: width_def} $\width_k(M) < W$,
		\item \label{item: width_cover} There exists an open cover $\mathcal{C}$ of $M$ such that 
		$\diam(\mathcal{C}) < W $ and  the multiplicity of $\mathcal{C}$ is at most $k+1$.
	\end{enumerate}
	\begin{proof}
	 We first prove $(\ref{item: width_def}) \Rightarrow (\ref{item: width_cover})$. Since $\width_k(M) < W$, it follows from  the definition of $\width_k(M)$ that   there exists a $k$-dimensional locally finite simplicial complex  $P^k$ and a continuous map  $f\colon M \rightarrow P^k$ such that $$\diam(f^{-1}(p)) < W- 2\varepsilon,$$ 
	 for any $p \in P$, where $\varepsilon$ is a sufficiently small positive constant. We see that for each   $p\in P$, there exists $r_p >0$ such that the $$\diam(f^{-1}(B(p,r_p))) < W-\varepsilon,$$ where $B(p, r_p) = \{q \in P : \dist(p,q) < r_p\}$. Note that $\{B(p,r_p)\}_{p \in M}$ forms an open  cover of $P$. Since the $\dim(P)=k$, the cover $\{B(p,r_p)\}_{p \in M}$ admits a refinement open cover  $\mathcal{C} = \{C_i\}_{i \in I}$ of $P$ such that the multiplicity of $\mathcal{C}$ is no greater than $k+1$. Since $\diam(f^{-1}(B(p,r_p))) < W - \varepsilon$ for any $p \in P$ and any member $C_i \in \mathcal{C}$ is contained in a member of $\{B(r,r_p)\}_{p \in M}$, the collection $\mathcal{C}_M = \{f^{-1}(C_i)\}_{i \in I}$ forms an open cover of $M$ with $$\diam(\mathcal{C}_M) \leq W-\varepsilon < W,$$ which is the desired open cover in (\ref{item: width_cover}).
	 
	 Now we prove $ (\ref{item: width_cover}) \Rightarrow (\ref{item: width_def})$. Suppose that $\mathcal{C} = \{C\}_{i\in I}$ is an open cover of $M$ such that the $\diam(\mathcal{C}) < W$ and the multiplicity of $\mathcal{C}$ is no greater than $k+1$. Let  $\mathcal{N}$ be the never associated with the cover $\mathcal{C}$ and $\varphi\colon M \rightarrow \mathcal{N}$ is the map constructed as in line (\ref{eq: map_to_nerve}). Then for each $p\in \mathcal N$,  $\diam(\varphi^{-1}(p)) < W$ since $\varphi^{-1}(p) \subset  C_i \subset \mathcal{C}$ for some $i\in I$. Moreover, $\mathcal{N}$ is a simplicial complex of dimension $k$.  It follows that  $\width_k(M) < W$.
	\end{proof}
\end{lemma}

\begin{remark}
	By comparing  Lemma \ref{thm: uryson_with_equivalent} with Remark \ref{rmk:asysmptotic}, we see that 
	\begin{equation}
	\dim_{\macro}(M) \leq \dim_{asy}(M)
	\end{equation}
	where $\dim_{asy}(M)$ is the asymptotic dimension of $(M, g)$.
\end{remark}
\begin{lemma}[cf. {\cite[4.C]{Gromov5fold}}]
	If $(M^n,g)$ is complete Riemannian manifold, then
	\begin{equation}\label{eq: fillrad_width}
		\fillrad(M) \leq \frac{1}{2} \width_{n-1}(M).
	\end{equation}
\end{lemma}
\begin{proof}
	Let $f\colon M \rightarrow P^{n-1}$ be a continuous map such that $$\diam(f^{-1}(p)) \leq \width_{n-1}(M) + 2\varepsilon .$$ 
	Choose $\omega =  (\width_{n-1}(M) + 2\varepsilon)/2$. Let 
	$$C_f = ([0,\omega]\times M) \cup P \,/\sim ,$$ be the mapping cylinder where 
	$(\omega,x)$ is identified with $f(x) $ for each $ x\in M $.
     We can extend the metric on $M$ to a metric on  $C_f$ such that the length of each path $\{x\} \times [0, \omega]$ is  $\omega$. Since $\dim(P) = n-1$, we conclude that $\fillrad(M) \leq \omega$. Since $\varepsilon$ is arbitrary, it follows that  $\fillrad(M) \leq \frac{1}{2}\width_{n-1}(M)$.
\end{proof}

Conversely, we propose the following conjecture.
\begin{conjecture} \label{conj: width_filling}
	There exists a universal constant $c_n$ such that 
	\begin{equation}
		\width_{n-1}(M) \leq c_n\fillrad(M)
	\end{equation}
	for all  complete Riemannian manifolds $(M^n,g)$. 
\end{conjecture}

Moreover, motivated by Llarull's theorem \cite{Llarull_sharp}, we introduce a metric invariant as follows.
\begin{definition}
	Given a   complete Riemannian manifold $(M^n,g)$, we define
	\begin{equation}
		\radsphere(M) = \inf_{f, R}\{R \mid f\colon  M \rightarrow \mathbb{S}^n(R) \textup{ with }\lip(f) \leq 1 \textup{ and } \deg(f) \neq 0\},
	\end{equation}
	where  $f$ is assumed to be locally constant near infinity, in particular, $\deg(f)$ is well-defined.
\end{definition}

\begin{lemma} \label{thm: radsphere_fillrad} If  $(M^n,g)$ is a complete Riemannian manifold, then 
			\begin{equation}
			\radsphere(M) \leq \frac{4}{\pi} \fillrad(M).
		\end{equation}
	\end{lemma}
	\begin{proof}
		For any small $\varepsilon >0$, we set $\mathbf r= \radsphere(M) - \varepsilon$. By the definition of $\radsphere(M)$,  there exists a map $f\colon  M \rightarrow \mathbb{S}^n(R)$ such that $\lip(f) \leq 1$ and $\deg(f) \neq  0$. Note that Lemma \ref{lemma:filling_radius_decreasing} implies that
		$$\fillrad(M) \geq \fillrad(\mathbb{S}^n(R)) $$
		It is known  that $\fillrad(\mathbb{S}^n(R)) = \frac{\mathbf r}{2}\arccos(-\frac{1}{n+1})$. See for example \cite[Page 8]{gromovfilling}. In particular, we have 
		$$\fillrad(M) \geq  \frac{\mathbf r}{2}\arccos(-\frac{1}{n+1}).$$
		Since $\varepsilon >0$ can be arbitrarily small, we have
		$$\fillrad(M) \geq \frac{\radsphere(M)}{2} \arccos(-\frac{1}{n+1}) .$$
		Note that  $ \displaystyle\frac{1}{2} \arccos(-\frac{1}{n+1}) \geq \frac{\pi}{4}$ for all $n \geq 1$. It follows that
				$$\fillrad(M) \geq \frac{\pi}{4} \radsphere(M).$$
	\end{proof}

Now we recall the definition of injectivity radius.

\begin{definition}
 Let $(M^n,g)$ be a  complete Riemannian manifold. 
	\begin{enumerate}
		\item  We define the injectivity radius of $M$ to be 	\[ \inj(M) =  \inf_{p \in M} \sup_{r} \{r \mid   \exp_p\textup{ is a diffeomorphism on } B(p, r) \subset T_pM \}.\]
		
		\item  We define the conjugate radius of $M$ to be 
		\begin{align*}
			 \conj(M) = \inf_{p \in M} \sup_r \{r\mid  \exp_p\textup{ is an immersion on } B(p, r) \subset T_pM \}.
		\end{align*}
		\item We define the convexity radius of $M$ to be 	\[ \conv(M) =  \inf_{p \in M} \sup_{r} \{r \mid  \textup{the metric ball }\mathcal B(p, r) \subset M \textup{ is convex}\}.\]
	\end{enumerate}
Here $\exp_p$ is the exponential map from the tangent space $T_pM$ to the manifold $M$ and $\mathcal{B}(p,r) = 
\{x \in M \mid  d_g(p,x) < r\}$ is the metric ball of radius $r$ centered at $p$. 
\end{definition} 	

It clear that  $$2 \conv(M) \leq \inj(M) \leq \conj(M).$$

\begin{lemma} \label{thm: inj_filling}
	If  $(M^n, g)$ is a complete Riemannian manifold, then 
		     \begin{equation}\label{thm: inj_fillrad}
		\inj(M) \leq (n+2) \fillrad(M) .
	\end{equation}
\end{lemma}
	\begin{proof}
		
		 Let $ R \leq \frac{\inj(M)}{n+2} - \delta$ for some small $\delta>0$. If  $P$  is an  $(n+1)$ dimensional locally finite  simplicial complex such that $M \subset P \subset N_{R}(M) \subset L^\infty(M)$, then we can subdivide $P$ into sufficiently small simplices such that any simplex $\triangle^{n+1} \subset P$, whose vertices are denoted by $p_0, p_1, \cdots, p_n$, satisfying $d(x_i, x_j) \leq 2\delta$ for all $0 \leq i, j \leq n+1$. In other words, we have extend the identity map $i\colon M \to M$ to a continuous map $f_0\colon M\cup P^{(0)} \to M$, where $P^{(0)}$ is the $0$-th skeleton  of $P$. From now on, we denote the $k$-th skeleton  of $P$ by $P^{(k)}$. Now we shall  inductively extend the map $i\colon M \to M$ to a continuous map $f\colon P \to M$. 
		 
		 Let $m_i\in M$ be a point that realizes the shortest distance between $p_i$ and $M$.  We connect $m_i$ and $m_j$ by a geodesic segment $\triangle^1(m_i, m_j)$. By mapping the $1$-simplex $\triangle^1(p_i, p_j)$ to $\triangle^1(m_i, m_j)$,  obtain the desired extension $f_1\colon M\cup P^{(1)} \to M$. 

		 Now we need to fill in each triangle formed by three edges, say,   $\triangle^1(m_0, m_1),  \triangle^1(m_1, m_2)$ and $\triangle^1(m_0, m_2)$.  First, in the previous step, it follows from the triangle inequality that   $d_g(m_i, m_j) < 2\delta + 2R$. By applying the triangle inequality again, we see that  each point on $\triangle_1(m_0, m_1)$ lies within distance $\frac{3(2R+2\delta)}{2}$ of $m_2$. It follows all three edges lie in the metric ball $\mathcal B(m_2, \frac{3(2R+2\delta)}{2})$ centered at $m_2$. Since  $\frac{3(2R+2\delta)}{2} < (n+2)(R+\delta)$, the exponential map $\exp_{m_2}$ is a diffeomorphism from $B(0, \frac{3(2R+2\delta)}{2})\subset T_{m_2}M $ to  $\mathcal B(m_2, \frac{3(2R+2\delta)}{2})$. By the contractibility of the ball, it follows that the map $f_1$ defined on  the three edges  $\triangle^1(p_0, p_1),  \triangle^1(p_1, p_2)$ and $\triangle^1(p_0, p_2)$ extends to the corresponding $2$-simplex $\Delta^2(p_0, p_1, p_2)$ in $P$. Denote the image of $\Delta^2(p_0, p_1, p_2)$ in $M$ by $\Delta^2(m_0, m_1, m_2)$.  Repeating this process for each triangle, we obtain the desired extension $f_2\colon M\cup P^{(2)} \to M$.

		 Now we want to fill in the nonsolid tetrahedron formed by, say,  $\Delta^2(m_0, m_1, m_2)$, $\Delta^2(m_1, m_2, m_3)$, $\Delta^2(m_0, m_1, m_3)$ and $\Delta^2(m_0, m_2, m_3)$. By using the triangle inequality, a similar argument as above shows that every point in $\Delta^2(m_0, m_1, m_2)$ lies within distance $2(2R+2\delta)$ of $m_3$. It follows all $\Delta^2(m_0, m_1, m_2)$, $\Delta^2(m_1, m_2, m_3)$, $\Delta^2(m_0, m_1, m_3)$ and $\Delta^2(m_0, m_2, m_3)$ lie in the metric ball $\mathcal B(m_3, 2(2R+2\delta))$ centered at $m_3$. Since $2(2R+2\delta) < (n+2)(R + \delta)$, the same argument in the previous inductive step shows that the map $f_2\colon \colon M\cup P^{(2)} \to M$ extends to a map $f_3\colon M\cup P^{(3)} \to M$.

		 By repeating the above steps inductively, we eventually obtain an extension of the identity map $i\colon M \to M$ to a continuous map $f\colon P \to M$, as long as $R \leq \frac{\inj(M)}{n+2} - \delta$.  
		This implies  $M$ can not bound any $(n+1)$ dimensional locally finite simplicial complex $P$ in $N_R(M)$. It follows that     
		\[\displaystyle \fillrad(M) \geq \frac{\inj(M)}{n+2}.\]
		This completes the proof. 
	\end{proof}

	To summarize, we have the following inequalities of the above metric invariants: 
	\begin{equation} \label{eq: series_of_metric_invariants}
		\frac{\inj(M)}{n+2} \leq \fillrad(M) \leq \width_{k}(M) \leq \diam(M), \textup{ for } 0\leq k \leq n-1,
	\end{equation}
\begin{equation} \label{eq: radshpere_fillingradius}
	\textup{ and } \radsphere(M) \leq \frac{4}{\pi} \fillrad(M).
\end{equation}

Motivated  by the above inequalities,  we propose the following conjecture.

\begin{conjecture} \label{conj: inj_radsphere}
	There exists a universal constant $c_n$ such that	
\begin{equation}
		\inj(M) \leq c_n \radsphere(M).
\end{equation}
for all complete Riemannian manifolds $(M^n,g)$. 
\end{conjecture}

\subsection{Metric invariants and uniformly positive scalar curvature}
In this subsection, we review some classical results and conjectures on the connections between (uniformly) positive scalar curvature and the  various metric invariants that we considered in the previous subsection. 

Let us  start with the maximal diameter theorem of Myers and Cheng, which states that strictly positive Ricci curvature can control the diameter of the space. More precisely,

\begin{theorem}[\cite{cheng_eigenvalue,myers_diameter}]\label{mdt}
	If $(M^n,g) $ is a complete Riemannian manifold with Ricci curvature $\Ric(g) \geq n-1$, then $\diam(M) \leq \pi$, and equality holds if and only if $M$ is the unit round sphere $\mathbb S^n(1)$. 
\end{theorem}
The above maximal diameter theorem is a rigidity type result in the context of Ricci curvature bounded from below. Intuitively speaking, it states  that uniformly positive Ricci curvature controls the diameter of the underlying space. Moreover, by applying the theorem to the universal cover, one sees that if a manifold admits  a complete Riemannian metric with uniformly positive Ricci curvature, then its fundamental group is finite. Generally speaking, scalar curvature, being the trace of the Ricci tensor, imposes much less constraints on the topology and geometry of the underlying space. However, in some cases surprisingly, rigidity type results persist in the context of scalar curvature bounded from below. See  \cite{Gromovadozen} and \cite{Gromov4lectures2019} for more systematic surveys on this subject. 

For example, the following theorem of Green shows that positive scalar curvature imposes constraints on the conjugate radius, hence  the injectivity radius.
\begin{theorem}[\cite{green_conjugate}] \label{green}
	Suppose that $(M^n,g)$ is a closed Riemannian manifold. If the scalar curvature  is uniformly positive $Sc(g) \geq n(n-1)$, then
	$
	\conj(M) \leq \pi,
	$
	and  equality holds if and only if $M$ is isometric to the unit round sphere $\mathbb{S}^n(1)$.
\end{theorem}
Since the injectivity radius satisfies $\mathrm{Inj}(M) \leq \mathrm{conj}(M)$, Theorem \ref{green}  implies that for any closed Riemannian manifold with positive scalar curvature $\Sc_g \geq n(n-1)$, we have
$
\mathrm{Inj}(M) \leq \pi,
$
and equality holds if and only $M$ is isometric to the unit round sphere $\mathbb{S}^n(1)$. Note that the manifold $M$ is assumed to be closed in Green's theorem.  If one assumes that $(M^n,g)$ is a complete, non-compact Riemannian manifold with uniformly positive scalar curvature $\Sc_g \geq n(n-1)$, then it is still an open question whether  $\inj(M) \leq \pi$.  In fact,  Gromov proposed the following stronger conjecture. 
\begin{conjecture}[\cite{gromov2023_injectivity_radius}]\label{conj: green_noncompact}
	Suppose that $(M^n,g)$ is a complete, non-compact Riemannian manifold with uniformly positive scalar curvature $\Sc(g) \geq (n-1)(n-2)$. Then we have 
		$$\inj(M) \leq \pi,$$
	and equality holds if and only if $M$ is isometric to  $(\mathbb{S}^{n-1}(1) \times \mathbb{R}, g_{\mathbb{S}^{n-1}} + dt^2)$.
\end{conjecture}

The following theorem is essentially a formulation of Llarull's theorem on scalar rigidity of round spheres \cite{Llarull_sharp}.


\begin{theorem}[\cite{Llarull_sharp}]
	Suppose that $(M^n,g)$ is a complete, spin, noncompact manifold with uniformly positive scalar curvature $\Sc_g \geq n(n-1)$. Then we have 
	$$\radsphere(M) \leq 1,$$
	and equality holds if and only if $(M,g)$ is isometric to the unit round sphere $\mathbb{S}^n(1)$.
\end{theorem}
In particular, if we assume Conjecture \ref{conj: inj_radsphere},  then the above theorem of Llarull will give a universal (but possibly non-optimal) upper bound on the injectivity radius, which would lead to a positive solution of a non-optimal version of  Conjecture \ref{conj: green_noncompact}. 

Moreover, in view of the inequalities from line \eqref{eq: series_of_metric_invariants} and \eqref{eq: radshpere_fillingradius}, the following Uryson width conjecture  of Gromov is closely related to Conjecture \ref{conj: green_noncompact}.  
\begin{conjecture} [\cite{gromovmetricstructure, Gromovadozen}]\label{conj: gromov_uryson_width_conjecture}
	There exists a universal constant $c_n$ such that
	\begin{equation}
		\width_{n-2}(M) \leq c_n
	\end{equation}
	for all complete Riemannian manifolds  $(M^n,g)$ with  uniformly positive scalar curvature $\Sc_g \geq n(n-1)$.
\end{conjecture}
An immediate consequence of Conjecture \ref{conj: gromov_uryson_width_conjecture} is the following conjecture of Gromov and Lawson. 
\begin{conjecture} [Gromov--Lawson]\label{conj: GLSY}
	If $M^n$ is a closed, aspherical manifold, then  $M$ admits no Riemannian metric with positive scalar curvature.
\end{conjecture}
Let us briefly comment on the current status of  the various conjectures that have appeared in the present paper, and the relations among them. It follows from the metric inequalities (\ref{eq: series_of_metric_invariants}) that  Conjecture \ref{conj: gromov_uryson_width_conjecture} implies Conjecture \ref{conj: gromov_filling_radius_conjecture}, where the latter together with Example \ref{example: ashperical_filling} implies Conjecture \ref{conj: GLSY}.
 For three dimensional manifolds,  Conjecture \ref{conj: gromov_uryson_width_conjecture} has been confirmed in \cite{gl_psc_dirac, Gromov4lectures2019, liokumovich2023waist}.
Moreover, Conjecture \ref{conj: GLSY} has been proved for dimensions $n=3, 4, 5$  \cite{gl_psc_dirac, Gromov5fold,Chodosh:2020tk}. In dimension higher than $5$, Conjecture \ref{conj: GLSY} remains open in general. On the other hand, by applying noncommutative geometric methods, Conjecture \ref{conj: GLSY} has been verified in all dimensions for a large class of aspherical manifolds whose fundamental group satisfies some geometric conditions.  For example, the third author proved Conjecture \ref{conj: GLSY} for all  $n$-dimensional aspherical manifolds whose fundamental group has finite asymptotic dimension or more generally whose fundamental group coarsely embeds into a Hilbert  space \cite{Yu_zero_spectrum, Yu_Novikov_fsd, Yucoarseembed}. Moreover, in conjunction with Conjecture \ref{conj: width_filling}, a weaker version of Conjecture \ref{conj: gromov_uryson_width_conjecture} that asserts  $\width_{n-1}(M) \leq c_n$ instead of $\width_{n-2}(M) \leq c_n$ is  equivalent to Conjecture \ref{conj: gromov_filling_radius_conjecture}. In this sense, the main theorem of the present paper could also be viewed as a solution to  this weaker version of Conjecture \ref{conj: gromov_uryson_width_conjecture} for the class of manifolds with finite asymptotic dimension.

\end{document}